\documentclass{amsart}
\usepackage{amssymb,amsmath,amsthm,amsbsy}
\usepackage[all]{xy}

\usepackage{hyperref}

\newtheorem{thm}{Theorem}[section]
\newtheorem{prop}[thm]{Proposition} 
\newtheorem{lemma}[thm]{Lemma}
\newtheorem{cor}[thm]{Corollary}

\theoremstyle{definition} 
\newtheorem{dfn}[thm]{Definition}

\theoremstyle{remark}
\newtheorem{rmk}[thm]{Remark}
\newtheorem{ex}[thm]{Example} 

\newtheorem*{rem}{Remark}

\newcommand{\la}[4]{
\xymatrix{#1 \ar[r] \ar@<2pt>[d] \ar@<-2pt>[d] & #2 \ar@<2pt>[d] \ar@<-2pt>[d] \\
#3 \ar[r] & #4}}
\newcommand{\dvb}[4]{
\xymatrix{#1 \ar[r] \ar[d] & #2 \ar[d] \\
#3 \ar[r] & #4}}

\newcommand{\pdiff}[2]{\frac{\partial #1}{\partial #2}}

\newcommand{\defequal}{:=}

\newcommand{\integers}{{\mathbb Z}}

\newcommand{\vect}{\mathfrak{X}}

\renewcommand{\bigwedge}{\mbox{{\Large{$\wedge$}}}}
\renewcommand{\tilde}[1]{\widetilde{#1}}
\renewcommand{\hat}[1]{\widehat{#1}}
\renewcommand{\bar}[1]{\overline{#1}}
\newcommand{\LA}{\mathcal{LA}}
\newcommand{\VB}{\mathcal{VB}}

\newcommand{\iso}{\stackrel{\sim}{\to}}

\newcommand{\cyclic}{\{\mbox{cycl.}\}}

\DeclareMathOperator{\im}{im}

\DeclareMathOperator{\rank}{rank}
\DeclareMathOperator{\ad}{ad}

\DeclareMathOperator{\End}{End}

\DeclareMathOperator{\str}{str}
\DeclareMathOperator{\coker}{coker}
\DeclareMathOperator{\Hom}{Hom}

\numberwithin{equation}{section}

\newcommand{\boundary}{\partial}

\newcommand{\Etot}{\mathcal{E}}
\newcommand{\starover}[1]{^*_{#1}}
\newcommand{\totaldiff}{\mathcal{D}}

\newcommand{\cs}{\mathrm{cs}}
\newcommand{\odd}{\mathrm{odd}}
\newcommand{\even}{\mathrm{even}}
\newcommand{\presup}[1]{\vphantom{ }^{#1}\!}
\newcommand{\id}{\mathrm{id}}
\newcommand{\dee}[1]{\mathrm{d} #1 \,}

\newcommand{\Wii}{\Gamma(A^* \otimes E^* \otimes C)}
\newcommand{\nue}[1]{\mathring{#1}}
\newcommand{\Ker}{\ker}
\newcommand{\coKer}{\coker}

\begin{document}
\title[$\VB$-algebroids and representation theory of Lie algebroids]{Lie algebroid structures on double vector bundles \\  and representation theory of Lie algebroids}
\author{Alfonso Gracia-Saz}
\address{Department of Mathematics\\
University of Toronto\\
40 Saint George Street, Room 6290\\
Toronto, Ontario, Canada M5S 2E4}
\email{alfonso@math.toronto.edu}

\author{Rajan Amit Mehta}
\address{Department of Mathematics\\
Washington University in Saint Louis\\
One Brookings Drive\\
Saint Louis, Missouri, USA 63130}
\email{raj@math.wustl.edu}

\begin{abstract}
  A $\VB$--algebroid is essentially defined as a Lie algebroid object in the category of vector bundles.  There is a one-to-one correspondence between $\VB$--algebroids and certain flat Lie algebroid superconnections, up to a natural notion of equivalence.  In this setting, we are able to construct characteristic classes, which in special cases reproduce characteristic classes constructed by Crainic and Fernandes.  We give a complete classification of regular $\VB$--algebroids, and in the process we obtain another characteristic class of Lie algebroids that does not appear in the ordinary representation theory of Lie algebroids.
\end{abstract}
\maketitle

\section{Introduction}

Double structures, such as double vector bundles, double Lie groupoids,  double Lie algebroids, and $\LA$--groupoids, have been extensively studied by Kirill Mackenzie and his collaborators \cite{mac:dbl2, mac:dblie2, mackenzie-2000, mac:duality, mackenzie-2006}.    In this paper, we study $\VB$--algebroids, which are essentially Lie algebroid objects in the category of vector bundles.  The notion of $\VB$--algebroids is equivalent to that of Mackenzie's $\LA$--vector bundles \cite{mac:dbl2}, which are essentially vector bundle objects in the category of Lie algebroids.  Our guiding principle is that $\VB$--algebroids may be viewed as generalized Lie algebroid representations.  

An obvious drawback of the usual notion of Lie algebroid representations is that there is no natural ``adjoint'' representation; for a Lie algebroid $A \to M$, the action of $\Gamma(A)$ on itself via the bracket is generally not $C^\infty(M)$-linear in the first entry.  One possible solution to this problem was given by Evens, Lu, and Weinstein \cite{elw}, in the form of representations ``up to homotopy''.  Briefly, a representation up to homotopy is an action of $\Gamma(A)$ on a $\integers_2$-graded complex of vector bundles, where the $C^\infty(M)$-linearity condition is only required to hold up to an exact term.  With this definition, they were able to construct an adjoint representation up to homotopy of $A$ on the ``$K$--theoretic'' formal difference $A \ominus TM$.  This representation up to homotopy was used by Crainic and Fernandes \cite{crainic:vanest, cf} to construct characteristic classes for a Lie algebroid, the first of which agrees up to a constant with the modular class of \cite{elw}.

Another notion of an adjoint representation was given by Fernandes \cite{fernandes}, who generalized Bott's theory \cite{bott:cc} of secondary, or ``exotic'', characteristic classes for regular foliations.  The key element in Fernandes's construction is the notion of a \emph{basic connection}, which, for a Lie algebroid $A \to M$, is a pair of $A$-connections on $A$ and $TM$, satisfying certain conditions.  These conditions imply that, although the individual connections generally have nonzero curvature, there is a sense in which they are flat on the formal difference; therefore, they can be used to produce secondary characteristic classes.  It was shown in \cite{cf} that these characteristic classes agree up to a constant with those constructed via the adjoint representation up to homotopy in \cite{crainic:vanest}.

Yet another generalized notion of Lie algebroid representation appears in Vaintrob's paper \cite{vaintrob} on the supergeometric approach to Lie algebroids.  There, a \emph{module} over a Lie algebroid $A$ is defined as a $Q$-vector bundle (i.e.\ vector bundle in the category of $Q$-manifolds) with base $A[1]$.  To our knowledge, this idea has not previously been explored in depth.

One can immediately see from the supergeometric perspective that a $\VB$--algebroid is a special case of a Lie algebroid module.  Thus, we may interpret the notion of $\VB$--algebroids as providing a description of certain Lie algebroid modules in the ``conventional'' language of brackets and anchors.  In particular, this special case includes Vaintrob's adjoint and coadjoint modules (see Example \ref{ex:lavb}).

As we see in Theorem \ref{thm:summary}, every $\VB$--algebroid may be noncanonically ``decomposed'' to give a flat Lie algebroid superconnection on a $2$-term complex of vector bundles, and conversely, one can construct a decomposed $\VB$--algebroid from such a superconnection.  We show in Theorem \ref{thm:vbtoflat} that different choices of decomposition correspond to superconnections that are equivalent in a natural sense; therefore we have the following key result: \emph{There is a one-to-one correspondence between isomorphism classes of $\VB$--algebroids and equivalence classes of $2$-term flat Lie algebroid superconnections.}

Given a Lie algebroid $A \to M$, a decomposition of the tangent prolongation $\VB$--algebroid $TA$ yields a flat $A$-superconnection on $A[1]\oplus TM$, the diagonal components of which form a basic connection in the sense of Fernandes \cite{fernandes}.  In fact, the $\VB$--algebroid $TA$ is a canonical object from which various choices of decomposition produce \emph{all} basic connections.  We interpret the $\VB$--algebroid $TA$ as playing the role of the adjoint representation.

As in the case of ordinary representations, one may obtain characteristic classes from $\VB$--algebroids; the construction of these classes is described in \S\ref{sec:cc}.  In the case of $TA$, one can see that our characteristic classes coincide with the Crainic-Fernandes classes.  

Then, we consider $\VB$--algebroids that are regular in the sense that the coboundary map in the associated $2$-term complex of vector bundles is of constant rank.  As in the case of representations up to homotopy (in the sense of Evens, Lu, Weinstein), there are canonical Lie algebroid representations on the cohomology $H(\Etot)$ of the complex.  However, we find that there is an additional piece of data---a canonical class $[\omega]$ in the 2nd Lie algebroid cohomology with values in degree $-1$ morphisms of $H(\Etot)$.  The two representations, together with $[\omega]$, completely classify regular $\VB$--algebroids.  

Finally, in the case of the $\VB$--algebroid $TA$, the class $[\omega]$ is an invariant of the Lie algebroid $A$.  More specifically, given a regular Lie algebroid $A$ with anchor map $\rho$, the cohomology class $[\omega]$ associated to the $\VB$--algebroid $TA$ is an element of $H^2(A; \Hom(\coker \rho, \ker \rho))$.  We see that $[\omega]$ may be interpreted as an obstruction to the regularity of the restrictions of $A$ to leaves of the induced foliation.

\begin{rem}
After submitting this paper, we learned of the work of Arias Abad and Crainic \cite{ari:rhla}, in which some of the constructions in this paper are developed independently.  In particular, they define the notion of \emph{representation up to homotopy}, which coincides with our notion of superrepresentation as in Definition \ref{dfn:asuper}, and which is different from the notion of representation up to homotopy according to Evens, Lu, and Weinstein \cite{elw}.  As we show in Theorem \ref{thm:summary}, a $\VB$--algebroid, after choosing a decomposition, corresponds to a superrepresentation in two degrees.  As an analogy, a $\VB$--algebroid is to a superrepresentation  in two degrees what a linear map is to a matrix, with choice of decomposition playing the role of choice of basis.  This is described in further detail for the example $TA$ (which both here and in \cite{ari:rhla} is interpreted as the adjoint representation) in \textsection \ref{sec:ta}.  
\end{rem}

\subsection*{Structure of the paper}
\begin{itemize}
  \item The objects that this paper deals with are \emph{double vector bundles} equipped with additional structures.  We begin in \S\ref{sec:dvb} by recalling the definition of a double vector bundle and describing some of its properties.
  \item In \S\ref{sec:dla} we introduce our main object: $\VB$--algebroids, and we present various equivalent structures.
  \item In \S\ref{sec:genconns} we give a one-to-one correspondance between isomorphism classes of $\VB$--algebroids and certain equivalence classes of flat superconnections, hence interpreting $\VB$--algebroids as ``higher'' Lie algebroid representations.
  \item In \S\ref{sec:cc} we define characteristic classes for every $\VB$--algebroid.
  \item In \S\ref{sec:classification} we classify all regular $\VB$--algebroids.
  \item Finally, in \S\ref{sec:ta} we use the results from \S\ref{sec:classification} in the case of the ``adjoint representation'' to associate a cohomology class to every regular Lie algebroid that has a geometric interpretation in terms of regularity around leaves induced by the algebroid foliation.
\end{itemize}

\subsection*{Acknowledgements}
We were partially supported by grants from Conselho Nacional de Desenvolvimento Cient\'{i}fico e Tecnol\'{o}gico (CNPq) and the Japanese Society for the Promotion of Science (JSPS).  We thank the Centre de Recerca Matem\`{a}tica and the Centre Bernoulli for their hospitality while this research was being done.  We also thank Eckhard Meinrenken.

\section{Background: Double vector bundles}\label{sec:dvb}

The main objects that this paper deals with are \emph{double vector bundles} (DVBs) equipped with additional structures.  Therefore we shall begin in this section by briefly recalling the definition of a DVB and describing some properties of DVBs that will be useful later.  For details and proofs, see \cite{mac:duality}.

A DVB is essentially a vector bundle in the category of vector bundles.  The notion of a DVB was introduced by Pradines \cite{pradines} and has since been studied by Mackenzie \cite{mac:duality} and Konieczna and Urba\'{n}ski \cite{kon-urb}.  Mackenzie has also introduced higher objects ($n$-fold vector bundles \cite{mac:duality}) and more general double structures ($\LA$-groupoids and double Lie algebroids \cite{mac:dbl2,mac:dblie2, mackenzie-2000,mackenzie-2006}).  Recently, Grabowski and Rotkiewicz \cite{grabowski} have studied double and $n$-fold vector bundles from the supergeometric point of view.  

Most of the material in this section has appeared in the above-referenced work of Mackenzie.

\subsection{Definition of DVB}

In order to define double vector bundle, we begin with a commutative square
\begin{equation}\label{dvb}
\xymatrix{D \ar^{q^D_B}[r] \ar_{q^D_A}[d] & B \ar^{q_B}[d] \\ A \ar^{q_A}[r] & M},
\end{equation}
where all four sides are vector bundles.  We wish to describe compatibility conditions between the various vector bundle structures.

We follow the notation of \cite{mac:duality}.  In particular, the addition maps for the two vector bundle structures on $D$ are $+_A: D \times_A D \to D$ and $+_B: D \times_B D \to D$.  The zero sections are denoted as $0^A: M \to A$, $0^B: M \to B$, $\tilde{0}^A: A \to D$, and $\tilde{0}^B: B \to D$.

We leave the proof of the following proposition as an exercise.

\begin{prop}\label{prop:dvb}
The following conditions are equivalent:
\begin{enumerate}
\item $q^D_B$ and $+_B$ are vector bundle morphisms over $q^A$ and the addition map $+: A \times_M A \to M$, respectively.
\item $q^D_A$ and $+_A$ are vector bundle morphisms over $q^B$ and the addition map $+: B \times_M B \to M$, respectively.
\item For all $d_1$, $d_2$, $d_3$, and $d_4$ in $D$ such that $(d_1,d_2) \in D \times_B D$, $(d_3,d_4) \in D \times_B D$, $(d_1,d_3) \in D \times_A D$, and $(d_2,d_4) \in D \times_A D$, the following equations hold:
	\begin{enumerate}
	\item $q^D_A(d_1 +_B d_2) = q^D_A(d_1) + q^D_A(d_2)$,
	\item $q^D_B(d_1 +_A d_3) = q^D_B(d_1) + q^D_B(d_3)$,
	\item $(d_1 +_B d_2) +_A (d_3 +_B d_4) = (d_1 +_A d_3) +_B (d_2 +_A d_4)$.
	\end{enumerate}
\end{enumerate}
\end{prop}

\begin{dfn}A \emph{double vector bundle}(DVB) is a commutative square (\ref{dvb}), where all four sides are vector bundles, satisfying the conditions of Proposition \ref{prop:dvb}.
\end{dfn}

\begin{rmk} A smooth map between vector bundles that respects addition is a vector bundle morphism.  For this reason, it is unnecessary to refer to scalar multiplication in condition (3) of Proposition \ref{prop:dvb}.  Alternatively, Grabowski and Rotkiewicz \cite{grabowski} have given an equivalent definition of DVBs only in terms of scalar multiplication; they also give an interesting interpretation in terms of commuting Euler vector fields.
\end{rmk}

\begin{rmk}
It is sometimes requested as part of the definition of DVB that the double projection $(q^D_A,q^D_B):D \to A \oplus B$ be a surjective submersion.  Grabowski and Rotkiewicz \cite{grabowski} proved that this is a consequence of the rest of the definition.
\end{rmk}

\subsection{The core of a DVB}\label{sec:core}
The structure of a DVB (\ref{dvb}) obviously includes two vector bundles, $A$ and $B$, over $M$, which are called the \emph{side bundles}.  There is a third vector bundle $C$, known as the \emph{core}, defined as the intersection of the kernels of the bundle maps $q^D_A$ and $q^D_B$.  Out of the three bundles $A$, $B$, and $C$, the core is special in that it naturally embeds into $D$.  In fact, it fits into the short exact sequence of double vector bundles
\begin{equation} \label{shortexact}
\raisebox{1.5pc}{\dvb{C}{M}{M}{M}}     
\xymatrix{\ar@{^{(}->}[r] &}  
\raisebox{1.5pc}{\dvb{D}{B}{A}{M}}  
\xymatrix {\ar@{->>}[r] &}
\raisebox{1.5pc}{\dvb{A \oplus B}{B}{A}{M}}
\end{equation}

Given vector bundles $A$, $B$, and $C$, there is a natural double vector bundle structure on $A \oplus B \oplus C$ with side bundles $A$ and $B$ and core $C$; this DVB is said to be \emph{decomposed}.

A section (in the category of double vector bundles) of (\ref{shortexact}) is equivalent to an isomorphism inducing the identity map on $A$, $B$, and $C$, between $D$ and the decomposed DVB $A \oplus B \oplus C$.  This isomorphism is called a \emph{decomposition} of $D$.  Grabowski and Rotkiewicz \cite{grabowski} proved that decompositions always exist locally (over open sets of $M$), and a \v{C}ech cohomology argument shows that decompositions exist globally.  In fact, the space of decompositions of $D$ is a nonempty affine space modelled on $\Gamma(A^* \otimes B^* \otimes C)$.  Hence, a section of (\ref{shortexact}) always exists, albeit noncanonically.

\subsection{Linear and core sections}

Consider a DVB as in (\ref{dvb}).  There are two special types of sections of $D$ over $B$, which we call \emph{linear} and \emph{core}\footnote{In \cite{mythesis}, the term \emph{vertical} was used instead of \emph{core}. However, it is now apparent that the present terminology is more appropriate.} sections. 
As we will use in various proofs in Appendix \ref{app:sth}, statements about sections of $D$ over $B$ can often be reduced to statements about linear and core sections.  

\begin{dfn}
A section $X \in \Gamma(D,B)$ is \emph{linear} if $X$ is a bundle morphism from $B \to M$ to $D \to A$.  The space of linear sections is denoted as $\Gamma_\ell(D,B)$.
\end{dfn}

The core sections arise from sections of the core bundle $C \to M$, in the following way.  Let $\alpha: M \to C$ be a section of the core.  The composition $\iota \circ \alpha \circ q^B$, where $\iota$ is the embedding of $C$ into $D$, is a map from $B$ to $D$ but is not a right inverse of $q^D_B$.  Instead, $\Gamma(C)$ is embedded into $\Gamma(D,B)$ by
\begin{equation}\label{eqn:coreembed}
  \alpha \in \Gamma(C) \mapsto \bar{\alpha} \defequal \iota \circ \alpha \circ q^B +_A \tilde{0}^B \in \Gamma(D,B).
\end{equation}

\begin{dfn} \label{dfn:coresec}
The space $\Gamma_C(D,B)$ of \emph{core sections} is the image of the map \eqref{eqn:coreembed}.
\end{dfn}
In the rest of this paper we will use the same notation for $\alpha$ and $\bar{\alpha}$ if there is no ambiguity.

Let $X$ be a section of $D$ over $B$.  We say that $X$ is \emph{$q$-projectible} (to $X_0$) if $X_0 \in \Gamma(A)$ and $q^D_A \circ X = X_0 \circ q^B$.  A linear section $X$ is necessarily $q$-projectible to its \emph{base section}.  All core sections are $q$-projectible to the zero section $0^A$.  Conversely, if $\alpha$ is $q$-projectible to $0^A$, then $\alpha$ is a core section if and only if the map $(\alpha -_A \tilde{0}^B)$ is constant on the fibres over $M$.

\begin{rmk}\label{rmk:coords}
It may be helpful to see a coordinate description of the linear and core sections.  Choose a decomposition $D \equiv A \oplus B \oplus C$, and choose local coordinates $\{x^i, b^i, a^i, c^i\}$, where $\{x^i\}$ are coordinates on $M$, and $\{b^i\}$, $\{a^i\}$, and $\{c^i\}$ are fibre coordinates on $B$, $A$, and $C$, respectively.  Let $\{A_i, C_i\}$ be the frame of sections over $B$ dual to the fibre coordinates $\{a^i, c^i\}$.  Then $X \in \Gamma(D,B)$ is linear if and only if it locally takes the form
\begin{equation}\label{eqn:locallinear}
X = f^i(x) A_i + g^i_j (x) b^j C_i,
\end{equation}
and $\alpha \in \Gamma(D,B)$ is core if and only if it locally takes the form
\begin{equation*}
\alpha = f^i(x) C_i.
\end{equation*}
\end{rmk}

\begin{ex}\label{ex:te}
A standard example of a DVB is
\begin{equation}\label{eqn:te}
\xymatrix{TE \ar[r] \ar[d] & E \ar[d] \\ TM \ar[r] & M},
\end{equation}
where $E \to M$ is a vector bundle.  The core, consisting of vertical vectors tangent to the zero section of $E \to M$, is naturally isomorphic to $E$.  The linear sections of $TE$ over $E$ are the linear vector fields, and the core sections are the fibrewise-constant vertical vector fields.
\end{ex}

It is possible to characterize morphisms of DVBs in terms of linear and core sections.  Let 
\begin{equation*}
  \dvb{D}{B}{A}{M} \, \mbox{ and } \, \dvb{D'}{B}{A'}{M}
\end{equation*}
be DVBs with cores $C$ and $C'$, respectively.  Let $F: D \to D'$ be a map that is linear over $B$, and let $F^\sharp: \Gamma(D,B) \to \Gamma(D',B)$ be the induced map of sections.
\begin{lemma}\label{lemma:lcbundle}
  Under the above conditions, the following are equivalent:
\begin{enumerate}
  \item $F^\sharp$ sends linear sections to linear sections and core sections to core sections.
\item $F$ is a morphism of vector bundles from $D \to A$ to $D' \to A'$.  In other words, $F$ is linear with respect to both the horizontal and vertical vector bundle structures, which is the definition of a morphism of double vector bundles.
\end{enumerate}
\end{lemma}
\begin{proof}
  The proof is a straightforward exercise in coordinates.
\end{proof}

\subsection{Horizontal lifts}\label{sec:horizontal}
Linear sections may be used to introduce a concept that is equivalent to that of a decomposition of a DVB, which will be useful later.

It is clear from the local description of linear sections (\ref{eqn:locallinear}) that the space $\Gamma_\ell (D,B)$ is locally free as a $C^\infty(M)$-module, with rank equal to $\rank(A) + \rank(B)\rank(C)$.  Therefore, $\Gamma_\ell(D,B)$ is equal to $\Gamma(\hat{A})$ for some vector bundle $\hat{A} \to M$.  

There is a short exact sequence of vector bundles over $M$
\begin{equation}\label{shortexact-ahat}
\xymatrix{0 \ar[r] & B^* \otimes C = \Hom(B,C) \ar^-i[r] & \hat{A} \ar^\pi [r] & A \ar[r] & 0}.
\end{equation}

\begin{dfn}\label{dfn:hor}
  A \emph{horizontal lift} of $A$ in $D$ is a section $h: A \to \hat{A}$ of the short exact sequence \eqref{shortexact-ahat}.
\end{dfn}

\begin{prop}
  There is a one-to-one correspondence between horizontal lifts and decompositions $D \iso A \oplus B \oplus C$.
\end{prop}
\begin{proof}
There is a natural horizontal lift in the case of a decomposed double vector bundle $A \oplus B \oplus C$.  Therefore, there is a map $\kappa$ from decompositions $D \iso A \oplus B \oplus C$ to horizontal lifts of $D$.  The spaces of decompositions and of horizontal lifts are both affine spaces modelled on $\Gamma(A^* \otimes B^* \otimes C)$.  The map $\kappa$ is affine, and the associated linear map is the identity.
\end{proof}

\begin{ex}
  For the DVB of Example \ref{ex:te}, a decomposition $TE \iso TM \oplus E \oplus E$ is the same thing as a linear connection on $E \to M$.  A horizontal lift, in the sense of Definition \ref{dfn:hor}, coincides in this case with the usual notion of a horizontal lift for $E \to M$.
\end{ex}

\section{Doubles for Lie algebroids and vector bundles}\label{sec:dla}

The main object of study in this paper consists of a double vector bundle with additional structure.  There are various equivalent ways to describe the additional structure, including $\LA$--vector bundles (\S\ref{sec:lavb}), $\VB$--algebroids (\S\ref{sec:vba}), and Poisson double vector bundles (\S\ref{sec:dlp}).  There are also interpretations in terms of differentials (\S\ref{sec:diff}) and supergeometry (\S\ref{sec:supermfld}).

\subsection{\texorpdfstring{$\LA$}{LA}-vector bundles}\label{sec:lavb}
An $\LA$--vector bundle is essentially a vector bundle in the category of Lie algebroids.  More precisely, it is a DVB
\begin{equation}\label{lavb}
\xymatrix{D \ar^{q^D_E}[r] \ar_{q^D_A}[d] & E \ar^{q^E}[d]\\ A \ar^{q^A}[r] & M},
\end{equation}
where the horizontal sides are Lie algebroids and the structure maps for the vertical vector bundle structures form Lie algebroid morphisms.  Specifically, if $q^D_A$ is an algebroid morphism, then there is an induced Lie algebroid structure on the fibre product $D \times_A D \to E \times_M E$, and we can ask that the addition map $+_A: D \times_A D \to D$ be an algebroid morphism.  The notion of an $\LA$--vector bundle is due to Mackenzie \cite{mac:dbl2}.

\begin{rmk} \label{rem:proj}
Consider a DVB of the form (\ref{lavb}).  Given a Lie algebroid structure on $D \to E$, there is at most one Lie algebroid structure on $A \to M$ such that $q^D_A$ is an algebroid morphism.  If such a Lie algebroid structure exists on $A$, then we may say that the Lie algebroid structure on $D \to E$ is \emph{$q$-projectible}.  Thus the definition of an $\LA$--vector bundle may be restated in the following way:
\end{rmk}

\begin{dfn}
  An \emph{$\LA$--vector bundle} is a DVB (\ref{lavb}) equipped with a $q$-projectible Lie algebroid structure on $D \to E$ such that the addition map $+_A: D \times_A D \to D$ is an algebroid morphism.
\end{dfn}

\begin{ex}\label{ex:lavb} {\ }
\begin{enumerate}
\item The DVB (\ref{eqn:te}) is an $\LA$--vector bundle, where $TE$ has the canonical tangent Lie algebroid structure over $E$.
\item Let $A \to M$ be a Lie algebroid.  Then
\begin{equation}  \label{ex:ta}
\xymatrix{T^*A \ar[r] \ar[d] & A^* \ar[d] \\ A \ar[r] & M}
\quad \textrm{and} \quad
\xymatrix{TA \ar[r] \ar[d] & TM \ar[d] \\ A \ar[r] & M}
\end{equation}
are $\LA$--vector bundles, where the Lie algebroid structure on $T^*A = T^*A^* \to A^*$ arises from the Poisson structure on $A^*$, and that on $TA \to TM$ is the \emph{tangent prolongation} of the Lie algebroid structure on $A \to M$.  We remark that the latter is in fact a double Lie algebroid and thus may be viewed as $\LA$--vector bundle in two different ways.  To avoid confusion, we will always present $\LA$--vector bundles so that the relevant Lie algebroid structures are on the horizontal sides.

In \S\ref{sec:genconns}, we will see that $\LA$--vector bundles may be viewed as higher representation of Lie algebroids.  From this point of view, the $\LA$--vector bundles in  \eqref{ex:ta} will play the roles of the coadjoint and the adjoint representation of $A$.
\end{enumerate}
\end{ex}

\subsection{\texorpdfstring{$\VB$}{VB}-algebroids}\label{sec:vba}
There is an alternative set of compatibility conditions for the Lie algebroid and vector bundle structures of (\ref{lavb}). Recall that the spaces of linear and core sections are denoted by $\Gamma_\ell(D,E)$ and $\Gamma_C(D,E)$, respectively.

\begin{dfn}\label{dfn:vbla}
A \emph{$\VB$--algebroid} is a DVB as in (\ref{lavb}), equipped with a Lie algebroid structure on $D \to E$ such that the anchor map $\rho_D: D \to TE$ is a bundle morphism over $A \to TM$ and where the bracket $[\cdot, \cdot]_D$ is such that
\begin{enumerate}
\item $[\Gamma_\ell(D,E), \Gamma_\ell(D,E)]_D \subseteq \Gamma_\ell(D,E)$,
\item $[\Gamma_\ell(D,E), \Gamma_C(D,E)]_D \subseteq \Gamma_C(D,E)$,
\item $[\Gamma_C(D,E), \Gamma_C(D,E)]_D = 0$.
\end{enumerate}
\end{dfn}
\begin{rmk}
Since the anchor map $\rho_D$ is automatically linear over $E$, the condition that it be linear over $A$ is equivalent to asking that $\rho_D$ be a morphism of DVBs from \eqref{lavb} to \eqref{eqn:te}.
\end{rmk}

Given a $\VB$--algebroid (\ref{lavb}), there is a unique map $\rho_A : A \to TM$ such that the diagram 
\begin{equation}\label{anchordiag}
\xymatrix{D \ar^{\rho_D}[r] \ar_{q^D_A}[d] & TE \ar^{Tq^E}[d] \\ A \ar^{\rho_A}[r] & TM}
\end{equation}
commutes, and $\rho_A$ is necessarily linear over $M$.  Furthermore, a bracket $[\cdot,\cdot]_A$ on $\Gamma(A)$ may be defined by the property that, if $X$ and $Y$ in $\Gamma_\ell(D,E)$ are $q$-projectible to $X_0$ and $Y_0$, repectively, then $[X,Y]_D$ is $q$-projectible to $[X_0,Y_0]_A$.  The map $\rho_A$ and the bracket $[\cdot, \cdot]_A$ together form a Lie algebroid structure on $A \to M$.  We leave the details as an exercise for the reader.

\begin{rmk}\label{rmk:gammagrade}
To provide some motivation for the the bracket conditions in Definition \ref{dfn:vbla}, we consider the DVB of Example \ref{ex:te}.  In this case, the Euler vector field $\varepsilon$ on $E$ induces a grading on the space of vector fields on $E$, where $X \in \vect(E)$ is homogeneous of degree $p$ if $[\varepsilon, X] = pX$.  Then the linear vector fields (which are the elements of $\Gamma_\ell(TE,E)$) are precisely those of degree $0$, and the fibrewise-constant vertical vector fields (which are the elements of $\Gamma_C(TE,E)$) are precisely those of degree $-1$.  In this example, the bracket conditions simply state that the Lie bracket respects the grading of vector fields.  The interpretation of the bracket conditions in terms of a grading is carried out in the general case in \S\ref{sec:diff}.
\end{rmk}

\subsection{Equivalence of \texorpdfstring{$\LA$}{LA}-vector bundles and \texorpdfstring{$\VB$}{VB}-algebroids}\label{sec:equivalence}

$\LA$--vector bundles and $\VB$--algebroids are both specified by the same type of data---a DVB of the form (\ref{lavb}), where $D \to E$ is equipped with a Lie algebroid structure satisfying certain compatibility conditions.  Both set of compatibility conditions imply that $A \to M$ is also a Lie algebroid.  The compatibility conditions for $\LA$--vector bundles require that the vertical vector bundles respect the horizontal Lie algebroid, whereas the compatibility conditions for $\VB$--algebroids require that the horizontal Lie algebroid respects the vertical vector bundles.  The following theorem states that the two sets of compatibility conditions are equivalent.

\begin{thm}\label{thm:lavbvbla}
A double vector bundle of the form (\ref{lavb}), where the top side is equipped with a Lie algebroid structure, satisfies the $\LA$--vector bundle compatibility conditions if and only if it satisfies the $\VB$--algebroid compatibility conditions.
\end{thm}

The proof of Theorem \ref{thm:lavbvbla} is given in Appendix \ref{app:sth}.

\begin{rmk}
 Since the notions of $\LA$--vector bundle and $\VB$--algebroid are equivalent, we could at this point discontinue the use of the term ``$\VB$--algebroid'' in favor of the previously-established term ``$\LA$--vector bundle''.  However, we will see that the constructions in \S\ref{sec:genconns} that form the heart of this paper directly utilize the conditions in Definition \ref{dfn:vbla}; in other words, this paper relies in an essential way on the $\VB$--algebroid point of view.  For this reason, we will continue to use the term ``$\VB$--algebroid''.
\end{rmk}

\begin{rmk}
  In the language of category theory, an $\LA$--vector bundle is essentially a vector bundle object in the category of Lie algebroids, and a $\VB$--algebroid is essentially a Lie algebroid object in the category of vector bundles.  In this sense, Theorem \ref{thm:lavbvbla} is an analogue of the following category-theoretic result: if $X$ and $Y$ are  algebraic categories, then an $X$ object in the category of $Y$ is equivalent to a $Y$ object in the category of $X$.  See, for instance, \cite{cwm}.
\end{rmk}

\subsection{Poisson double vector bundles}\label{sec:dlp}

Again, consider a double vector bundle of the form (\ref{lavb}), where $D \to E$ is a Lie algebroid.  If we dualize $D$ over $E$, we obtain a new DVB
\begin{equation}\label{dualdvb}
\xymatrix{D\starover{E} \ar[r] \ar[d] & E \ar[d] \\ C^* \ar[r] & M},
\end{equation}
where the core is $A^*$.  For a discussion on the dualization of DVBs, see \cite{mac:duality}.

The algebra of functions on $D\starover{E}$ has a canonical double-grading; we denote by $C^\infty_{p,q}(D\starover{E})$ the space of functions that are homogeneous of degrees $p$ and $q$ over $E$ and $C^*$, respectively.  The space $C^\infty_{1,\bullet}(D\starover{E})$ of functions that are linear over $E$ may be identified with $\Gamma(D,E)$, and it is clear from the coordinate description of Remark \ref{rmk:coords} that, under this identification, we have $\Gamma_\ell(D,E) = C^\infty_{1,1}(D\starover{E})$ and $\Gamma_C(D,E) = C^\infty_{1,0}(D\starover{E})$.  Note that the grading on $\Gamma(D,E)$ that is induced from the identification with $C^\infty_{1,\bullet}(D\starover{E})$ is not the same as the grading described in Remark \ref{rmk:gammagrade}, but is shifted by $1$.

As usual, the Lie algebroid structure on $D \to E$ induces a Poisson structure on $D\starover{E}$ that is linear over $E$, in the sense that $\{C^\infty_{p,\bullet}(D\starover{E}),C^\infty_{p',\bullet}(D\starover{E})\} \subseteq C^\infty_{p+p'-1,\bullet}(D\starover{E})$.  It is then fairly easy to see that the compatibility conditions of Definition \ref{dfn:vbla} are equivalent to the condition that $\{C^\infty_{\bullet,q}(D\starover{E}),C^\infty_{\bullet,q'}(D\starover{E})\} \subseteq C^\infty_{\bullet,q+q'-1}(D\starover{E})$.  In other words, we have

\begin{thm}\label{thm:dlp}
A double vector bundle of the form (\ref{lavb}), where the top side is equipped with a Lie algebroid structure, satisfies the $\VB$--algebroid compatibility conditions if and only if the induced Poisson structure on $D\starover{E}$ is linear over $C^*$.
\end{thm}

Mackenzie \cite{mackenzie-2006} has defined a \emph{Poisson double vector bundle} to be a DVB whose total space is equipped with a Poisson structure that is linear over both side bundles.  Theorem \ref{thm:dlp} states that there is a correspondence between Poisson double vector bundles and $\VB$--algebroid structures.  We note that this result was already established in \cite{mackenzie-2006}.

An interesting feature of Poisson double vector bundles is that the definition is symmetric with respect to the roles of the two side bundles.  On the other hand, the correspondence of Theorem \ref{thm:dlp} is not symmetric, which implies that there are in fact two different $\VB$--algebroid structures associated to each Poisson double vector bundle.  Using both correspondences, we are able to associate a $\VB$--algebroid structure on $D$ with a $\VB$--algebroid structure on $(D\starover{E})\starover{C^*}$; the latter is canonically isomorphic to $D\starover{A}$, so we obtain the following duality result: 
\begin{cor}
A $\VB$--algebroid structure on (\ref{lavb}) induces a dual $\VB$--algebroid structure on
\begin{equation}\label{dualdvb2}
\dvb{D\starover{A}}{C^*}{A}{M}.
\end{equation}
\end{cor}

\begin{rmk}
  In general, a DVB \eqref{lavb} and its two neighbors \eqref{dualdvb} and \eqref{dualdvb2} fit together to form a triple vector bundle 
\begin{equation}\label{eq:cube}
  \xymatrix@=5pt{ T^*D \ar[rr] \ar[dd] \ar[rd] && D^*_E \ar[rd] \ar'[d][dd] \\ & D \ar[rr] \ar[dd] && E \ar[dd] \\ D^*_A \ar[rd] \ar'[r][rr] && C^* \ar[rd] \\ & A \ar[rr] && M}
\end{equation}
If \eqref{lavb} is a $\VB$--algebroid, then the triple vector bundle \eqref{eq:cube} has the following structures:
\begin{itemize}
  \item The four horizontal edges are Lie algebroids.
\item The right and left faces are Poisson double vector bundles.
\item The other four faces are $\VB$--algebroids.
\end{itemize}
All of the above structures and their relations may be summarized as follows: the cube \eqref{eq:cube} is a Lie algebroid object in the category of Poisson double vector bundles.
\end{rmk}

\begin{ex}Let $E \to M$ be a vector bundle.  The $\VB$--algebroid (\ref{eqn:te}) is associated to the Poisson double vector bundle
\begin{equation*}
\xymatrix{T^*E \ar[r] \ar[d] & E \ar[d] \\ E^* \ar[r] & M},
\end{equation*}
from which we may obtain the dual $\VB$--algebroid
\begin{equation*}
\xymatrix{TE^* \ar[r] \ar[d] & E^* \ar[d] \\ TM \ar[r] & M}.
\end{equation*}
\end{ex}

\begin{ex}Let $A \to M$ be a Lie algebroid.  The $\VB$--algebroid $T^\star A$ in \eqref{ex:ta}
is associated to the Poisson double vector bundle
\begin{equation*}
\xymatrix{TA^* \ar[r] \ar[d] & A^* \ar[d] \\ TM \ar[r] & M},
\end{equation*}
and is dual to the $\VB$--algebroid $TA$ in \eqref{ex:ta}.
\end{ex}

\subsection{Compatibility in terms of algebroid differentials}\label{sec:diff}

Once again, consider a double vector bundle of the form (\ref{lavb}), where $D \to E$ is a Lie algebroid.  We may identify $\Gamma(D\starover{E}, E)$ with the space of functions on $D$ that are linear over $E$.  Furthermore, this identification gives $\Gamma(D\starover{E}, E)$ a grading according to polynomial degree over $A$, and this grading may be extended to $\bigwedge \Gamma(D\starover{E}, E)$.  We denote by $\Omega^{p,q}(D)$ the subspace of $\bigwedge^p \Gamma(D\starover{E}, E)$ consisting of those $p$-forms that are of degree $q$ over $A$.

Recall that in \S\ref{sec:dlp}, the space of sections $\Gamma(D,E)$ was given a grading, where the linear sections were of degree $1$ and the core sections were of degree $0$.  The gradings on $\Gamma(D,E)$ and $\Omega(D)$ agree up to a shift, in the sense that, for a degree $q$ section $X \in \Gamma(D,E)$, the operator $\iota_X$ on $\Omega(D)$ is of degree $q - 1$.

The Lie algebroid structure on $D \to E$ induces a differential $d_D$ on $\bigwedge \Gamma(D\starover{E},E)$.

\begin{thm}\label{thm:diff}
A double vector bundle of the form (\ref{lavb}), where the top side is equipped with a Lie algebroid structure, satisfies the $\VB$--algebroid compatibility conditions if and only if $d_D$ is of degree $0$ with respect to the ``over $A$  grading''.
\end{thm}
\begin{proof} 
Applying Lemma \ref{lemma:lcbundle} to the anchor map $\rho_D$, we have that the compatibility condition for the anchor is equivalent to the condition that, for a degree $q$ section $X \in \Gamma(D,E)$, the degree of $\rho_D (X)$ is $q-1$ as an operator on $C^\infty(E)$.  Additionally, the compatibility conditions for the bracket are equivalent to the condition that, for sections $X$ and $Y$ of degrees $q$ and $q'$, respectively, the degree of $\iota_{[X,Y]_D}$ is $q+q'-2$.

For $\omega \in \Omega^{p,q}(D)$ and $X_i \in \Gamma(D,E)$ of degree $q_i$ for $i=0,\dots,p$, the differential $d_D$ is given by the formula
\begin{equation}\label{eqn-diff}
\begin{split}
\iota_{X_p} \cdots \iota_{X_0} d_D \omega =& \sum_{i=0}^p (-1)^i \rho_D(X_i)\iota_{X_p} \cdots \hat{\iota_{X_i}} \cdots \iota_{X_0} \omega \\
& + \sum_{j>i} (-1)^j \iota_{X_p} \cdots \iota_{[X_i,X_j]_D} \hat{\iota_{X_j}} \cdots \hat{\iota_{X_i}} \cdots \iota_{X_0} \omega,
\end{split}
\end{equation}
Each term on the right hand side is of degree $q + \sum (|q_i|-1)$.  The left hand side must be of the same degree, which implies that $d_D$ is of degree $0$.
\end{proof}

\subsection{Supergeometric interpretation}\label{sec:supermfld}

It was observed by Vaintrob \cite{vaintrob} that the differential point of view for a Lie algebroid is more naturally stated in the language of supergeometry in the following way: a Lie algebroid structure on $A \to M$ is equivalent to a degree $1$ homological vector field on the graded manifold $A[1]$.  Here, $A[1]$ is the graded manifold whose algebra of ``functions'' is $\bigwedge \Gamma(A^*)$, and the operator $d_A$, as a derivation of this algebra, is viewed as a vector field on $A[1]$.  The modifier \emph{homological} indicates that $d_A^2 = 0$.

In the case of a $\VB$--algebroid \eqref{lavb}, we may form the graded manifold\footnote{Since $D$ is the total space of two different vector bundles, we use the subscript in $[1]_E$ to indicate that we are applying the functor $[1]$ to the vector bundle $D\to E$, as opposed to $D \to A$.} $D[1]_E$, whose algebra of ``functions'' $C^\infty(D[1]_E)$ is $\Omega^{\bullet,\bullet}(D)$.  The operator $d_D$ is viewed as a homological vector field on $D[1]_E$.

The algebra $C^\infty(D[1]_E)$ has a natural double-grading arising from the DVB structure
\begin{equation*}
\xymatrix{D[1]_E \ar[r] \ar[d] & E \ar[d] \\ A[1] \ar[r] & M},
\end{equation*}
and this double-grading coincides with the double-grading of $\Omega(D)$ that was introduced in \S\ref{sec:diff}.  In this point of view, we may use Theorem \ref{thm:diff} to effectively restate the definition of a $\VB$--algebroid as follows:

\begin{thm}
A $\VB$--algebroid structure on a DVB (\ref{dvb}) is equivalent to a vector field $d_D$ on $D[1]_E$ of bidegree $(1,0)$ such that $d_D^2 = 0$.
\end{thm}

\section{(Super)connections and horizontal lifts}\label{sec:genconns}
One of the main goals of this paper is to show that isomorphism classes of $\VB$--algebroids are in one-to-one correspondence with flat Lie algebroid superconnections up to a certain notion of equivalence.  With this in mind, we can understand a $\VB$--algebroid as a generalization of a Lie algebroid representation.  

Consider a $\VB$--algebroid (\ref{lavb}) with core $C$.  In \S\ref{sec:biga}-\S\ref{sec:bigsideandcore}, we will see that there is a natural Lie algebroid structure on $\hat{A}$, and that $\hat{A}$ possesses natural representations on $C$ and $E$.  In \S\ref{sec:sideandcore}, we use horizontal lifts to obtain $A$-connections on $C$ and $E$; unfortunately, the procedure is not canonical, and the induced connections are not flat.  However, the induced $A$-connections form part of a flat $A$-superconnection on a graded bundle (\S\ref{sec:superconn}).  Although the flat $A$-superconnection is noncanonical, different choices of horizontal lifts lead to superconnections that are equivalent in a way that will be described in \S\ref{sec:doDiad} and will be used in \S\ref{sec:classification} to classify regular $\VB$--algebroids.

\subsection{The fat algebroid} \label{sec:biga}
Consider a $\VB$--algebroid (\ref{lavb}) with core $C$.  As in \S\ref{sec:horizontal}, let $\hat{A}$ denote the vector bundle over $M$ whose space of sections is $\Gamma(\hat{A}) = \Gamma_\ell (D,E)$.  The bundle $\hat{A}$ has a natural Lie algebroid structure with bracket $[\cdot,\cdot]_{\hat{A}}$ and anchor $\rho_{\hat{A}}$ given by
\begin{equation*}
\begin{split}
[X,Y]_{\hat{A}} &= [X,Y]_D \\
\rho_{\hat{A}} (X) &= \rho_A (X_0),
\end{split}
\end{equation*}
where $X$ $q$-projects to $X_0$.  We refer to $\hat{A}$ as the \emph{fat algebroid}.

The projection map $\hat{A} \to A$ is a Lie algebroid morphism, the kernel of which may be identified with $\Hom(E,C)$.  Therefore $\Hom(E,C)$ inherits a Lie algebroid structure so that 
\begin{equation}\label{shortexact-biga}
\xymatrix{0 \ar[r] & \Hom(E,C) \ar^-i[r] & \hat{A} \ar^\pi [r] & A \ar[r] & 0}
\end{equation}
is an exact sequence of Lie algebroids over $M$.  

\begin{ex}\label{ex:fatjet}
When $D = TA$, where $A \to M$ is a Lie algebroid, then $\hat{A}$ is equal to the first jet bundle $J^1 A$ of $A$.  In this case, Crainic and Fernandes \cite{cf} have described natural representations of $J^1 A$ on $A$ and $TM$.  In \S\ref{sec:bigsideandcore}, we will extend this process to all $\VB$--algebroids.
\end{ex}

\subsection{The core-anchor}\label{sec:coreanchor}

To explicitly describe the Lie algebroid structure on $\Hom(E,C)$ inherited from \eqref{shortexact-biga}, it is useful to introduce an auxilliary map.  Since the anchor $\rho_D$ is a morphism of DVBs from
\begin{equation*}
  \dvb{D}{E}{A}{M} \, \mbox{ to } \, \dvb{TE}{E}{TM}{M},
\end{equation*}
it induces a linear map of the core vector bundles.

\begin{dfn}\label{dfn:coreanchor}
  The \emph{core-anchor} $\boundary$ of a $\VB$--algebroid \eqref{lavb} is minus the vector bundle morphism induced by the anchor map $\rho_D$ from the core $C$ of $D$ to the core $E$ of $TE$.
\end{dfn}

The core-anchor $\boundary$ is explicitly given by the equation
\begin{equation*}
\left\langle \boundary\alpha, e\right\rangle = -\rho_D(\alpha)(e)
\end{equation*}
for all $\alpha \in \Gamma(C)$ and $e \in \Gamma(E^*)$.  In other words, since $-\rho_D(\alpha)$ is a fibrewise-constant vertical vector field on $E$, it may be identified with a section $\boundary \alpha$ of $E$.

The map $\boundary$ is $C^\infty(M)$-linear and therefore is an element of $\Hom(C,E)$.  The bracket on $\Hom(E,C)$ is then given by
\begin{equation}\label{eqn:bracketphi}
[\phi, \phi'] = \phi \boundary \phi' - \phi' \boundary \phi
\end{equation}
for $\phi, \phi' \in \Hom(E,C)$.  The anchor is trivial, so $\Hom(E,C)$ is actually a bundle of Lie algebras.

\begin{ex}
  In the $\VB$--algebroid \eqref{ex:ta}, the core-anchor maps $A$ to $TM$ and is equal to minus the anchor of the algebroid $A \to M$.
\end{ex}

\subsection{Side and core representations of \texorpdfstring{$\hat{A}$}{the fat algebroid}}
\label{sec:bigsideandcore}
The fat algebroid has natural representations (i.e.\ flat connections) $\psi^c$ and $\psi^{s^*}$ on $C$ and $E^*$, respectively, given by
\begin{align}
\psi^c_\chi(\alpha) &\defequal [\chi,\alpha]_D, \\
\psi^{s^*}_\chi(e) &\defequal \rho_D(\chi)(e), \label{eqn:psisdual}
\end{align}
for $\chi \in \Gamma(\hat{A})$, $\alpha \in \Gamma(C)$, and $e \in \Gamma(E^*)$.  In \eqref{eqn:psisdual}, we view $e$ as a linear function on $E$.  Since $\rho_D(\chi)$ is a linear vector field on $E$, it acts on the space of linear functions.  

As usual, the representation $\psi^{s^*}$ may be dualized to a representation $\psi^s$ on $E$, given by the equation
\begin{equation*}
\langle \psi^s_\chi(\varepsilon), e \rangle \defequal \rho_{\hat{A}}(\chi)\langle \varepsilon, e\rangle - \langle \varepsilon, \psi^{s^*}_\chi (e)\rangle
\end{equation*}

We leave the following as an exercise.
\begin{prop}\label{prop:reprelations}
The representations $\psi^c$ and $\psi^s$ are related in the following ways:
\begin{enumerate}
\item $\boundary \psi^c_\chi = \psi^s_\chi \boundary$
\item $\phi \psi^s_\chi - \psi^c_\chi \phi = [\phi,\chi]_{\hat{A}}$
\end{enumerate}
for all $\chi \in \Gamma(\hat{A})$ and $\phi \in \Hom(E,C)$.
\end{prop}

The representations of $\hat{A}$ may be pulled back to obtain representations $\theta^c$ and $\theta^s$ of $\Hom(E,C)$ on $C$ and $E$. Explicitly, these represetnations are given by
\begin{align}
\label{eqn:rephomcore}\theta^c_\phi (\alpha) &= \phi \circ \boundary (\alpha),\\
\label{eqn:rephomside}\theta^s_\phi (\varepsilon) &= \boundary \circ \phi (\varepsilon),
\end{align}
for $\phi \in \Hom(E,C)$, $\alpha \in \Gamma(C)$, and $\varepsilon \in \Gamma(E)$.

We would like to be able to push the side and core representations of $\hat{A}$ forward to obtain representations of $A$; however, this is generally not possible since the induced representations of $\Hom(E,C)$ in (\ref{eqn:rephomcore})-(\ref{eqn:rephomside}) are nontrivial.  If $\boundary$ is of constant rank, the side and core $\hat{A}$-representations do induce $A$-representations $\nabla^K$ on the subbundle $K \defequal \ker \boundary$ and $\nabla^\nu$ on the quotient bundle $\nu \defequal \coker \boundary$.  These induced $A$-representations play an important role in the classification of regular $\VB$--algebroids in \S\ref{sec:classification}.

Even if $\boundary$ is not of constant rank, it is possible to noncanonically extend the (possibly singular) representations on $K$ and $\nu$ to $C$ and $E$, at the cost of introducing curvature.  We discuss this in the following section.

\subsection{Side and core \texorpdfstring{$A$}{A}-connections}\label{sec:sideandcore}

There does not exist in general a section of the short exact sequence (\ref{shortexact-biga}) in the category of Lie algebroids.  Nonetheless, sections in the category of vector bundles do exist; in \S\ref{sec:horizontal}, they were called \emph{horizontal lifts}.

Let us choose a horizontal lift $h: A \to \hat{A}$.  For $X \in \Gamma(A)$, we denote its image in $\Gamma(\hat{A})$ by $\hat{X} = h(X)$.  We may use $h$ to pull back the representations $\psi^c$ and $\psi^s$ to $A$-connections $\nabla^c$ and $\nabla^s$, respectively, so that, 
\begin{align}\label{eqn:scconn}
\nabla^c_X \defequal \psi^c_{\hat{X}}, && \nabla^s_X \defequal \psi^s_{\hat{X}}.
\end{align}

\begin{rmk}
  The connections $\nabla^s$ on $E$ and $\nabla^c$ on $C$ depend on the choice of horizontal lift.  However, they are extensions of the canonical flat connections $\nabla^K$ on the subbundle $K$ of $C$ and $\nabla^\nu$ on the quotient bundle $\nu$ of $E$, which were introduced above.
\end{rmk}

The induced side and core connections \eqref{eqn:scconn} will generally have nonzero curvature, resulting from the failure of $h$ to respect Lie brackets.  To this end, we define $\Omega \in \bigwedge^2 \Gamma(A^*) \otimes \Hom(E,C)$ as follows:
\begin{equation*}
\Omega_{X,Y} \defequal \hat{[X,Y]} - [\hat{X},\hat{Y}]
\end{equation*}
for $X,Y \in \Gamma(A)$.  We will later require the following identity:
\begin{lemma}\label{lemma:omegacycl}
For all $X,Y,Z \in \Gamma(A)$, 
\begin{equation*}
\Omega_{[X,Y],Z} + [\Omega_{X,Y},\hat{Z}] + \cyclic = 0.
\end{equation*}
\end{lemma}
\begin{proof}From the definition of $\Omega$, we have that
\begin{equation*}
\hat{[[X,Y],Z]} = \Omega_{[X,Y],Z} + [\Omega_{X,Y}, \hat{Z}] + [[\hat{X},\hat{Y}], \hat{Z}].
\end{equation*}
The result follows from the Jacobi identity.
\end{proof}

A direct computation reveals that the curvatures $F^c$ and $F^s$ of $\nabla^c$ and $\nabla^s$, respectively, satisfy the following equations for $X,Y \in \Gamma(A)$:
\begin{align}
\label{eqn:curvc}F^c_{X,Y} &= \theta^c_{\Omega_{X,Y}} = \Omega_{X,Y} \circ \boundary, \\
\label{eqn:curvs}F^s_{X,Y} &= \theta^s_{\Omega_{X,Y}} = \boundary \circ \Omega_{X,Y}.
\end{align}
Additionally, the following properties are immediate consequences of Proposition \ref{prop:reprelations}:
\begin{align}
\label{eqn:relconn}\boundary \circ \nabla^c_X &= \nabla^s_X \circ \boundary, \\
\label{eqn:relconnphi}\phi \circ \nabla^s_X - \nabla^c_X \circ \phi &= [\phi, \hat{X}],
\end{align}
for $X \in \Gamma(A)$ and $\phi \in \Hom(E,C)$.

\subsection{The \texorpdfstring{$A$}{A}-Superconnection}\label{sec:superconn}
So far, given a $\VB$--algebroid equipped with a horizontal lift, we have obtained the following data:
\begin{itemize}
  \item a bundle map $\boundary: C \to E$,
\item (in general, nonflat) $A$-connections $\nabla^c$ and $\nabla^s$ on $C$ and $E$, respectively, and
\item a $\Hom(E,C)$-valued $A$-$2$-form $\Omega$.
\end{itemize}
In this section, we will show that the above data may be combined to form a flat $A$-superconnection.  Let us first recall the definitions.

Let $A\to M$ be a Lie algebroid, let $\Omega(A)$ denote the algebra of Lie algebroid forms $\bigwedge \Gamma(A^*)$, and let $\mathcal{E}$ be a $\integers$-graded vector bundle over $M$.  The algebra $\Omega(A)$ and the space $\Gamma(\mathcal{E})$ are both naturally $\integers$-graded. We consider the space of $\mathcal{E}$-valued $A$-forms
 $ \Omega(A) \otimes_{C^\infty(M)} \Gamma(\mathcal{E})$
to be $\integers$-graded with respect to the total grading.

\begin{dfn}\label{dfn:asuper}
An \emph{$A$-superconnection} on $\mathcal{E}$ is an odd operator $\totaldiff$ on $\Omega(A) \otimes \Gamma(\mathcal{E})$ such that
\begin{equation}\label{eqn:superleibniz}
  \totaldiff(\omega s) = (d_A \omega)s + (-1)^p \omega \wedge \totaldiff(s)
\end{equation}
for all $\omega \in \Omega(A)$ and $s \in \Gamma(\mathcal{E})$, where $p$ is the degree of $\omega$.  We say that $\totaldiff$ is \emph{flat} if the \emph{curvature} $\totaldiff^2$ is zero.
\end{dfn}

\begin{rmk}
  When the graded bundle $\mathcal{E}$ is concentrated in degree $0$, Definition \ref{dfn:asuper} agrees with the notion of an $A$-connection in the sense of Fernandes \cite{fernandes}.  On the other hand, when $A = TM$, the above notion of an $A$-superconnection reduces to that of a superconnection in the sense of Quillen \cite{quillen:superconnections}.
\end{rmk}

\begin{rmk}
  The superconnections of primary interest in this paper are of degree $1$.  For this reason, in the remainder of this paper, by ``superconnection'' we will mean ``degree $1$ superconnection'' unless otherwise stated.  
\end{rmk}

Let us now return to the situation of a $\VB$--algebroid equipped with a horizontal lift $A \to \hat{A}$.  Let $D^c$ be the degree $1$ operator on $\Omega(A) \otimes \Gamma(C)$ associated to the core connection $\nabla^c$.  Similarly, let $D^s$ be the operator on $\Omega(A) \otimes \Gamma(E)$ associated to $\nabla^s$.  We may extend both $D^c$ and $D^s$ to $\Omega(A) \otimes \Gamma(C \oplus E)$ by setting $D^c (\omega \varepsilon) = D^s (\omega \alpha) = 0$ for all $\omega \in \Omega(A)$, $\varepsilon \in \Gamma(E)$, and $\alpha \in \Gamma(C)$.  

We may also view $\boundary$ and $\Omega$ as operators on $\Omega(A) \otimes \Gamma(C \oplus E)$, where for $\omega \in \Omega^p(A)$, $\alpha \in \Gamma(C)$, and $\varepsilon \in \Gamma(E)$,
\begin{align*}
\boundary(\omega \alpha) &= (-1)^p \omega \cdot \boundary(\alpha), && \boundary(\omega \varepsilon) = 0, \\
\Omega(\omega \alpha) &= 0, && \Omega(\omega \varepsilon) = (-1)^p \omega \wedge \Omega(\varepsilon).
\end{align*}

Although $\boundary$ and $\Omega$ are of degree $0$ and $2$, respectively, as operators on $\Omega(A) \otimes \Gamma(C \oplus E)$, they may both be viewed as degree $1$ operators on $\Omega(A) \otimes \Gamma(C[1] \oplus E)$, where the $[1]$ denotes that sections of $C$ are considered to be of degree $-1$.  Thus $\totaldiff \defequal \boundary + D^c + D^s + \Omega$ is a degree $1$ operator on $\Omega(A) \otimes \Gamma(C[1] \oplus E)$.  Clearly, $\totaldiff$ satisfies \eqref{eqn:superleibniz}, so $\totaldiff$ is a degree $1$ $A$-superconnection on $C[1] \oplus E$.

\begin{thm}The superconnection $\totaldiff$ is flat.\end{thm}
\begin{proof}
Let $F \defequal \totaldiff^2$ be the curvature of $\totaldiff$.  Since $\End(C[1] \oplus E)$ is concentrated in degrees $-1$, $0$, and $1$, we may decompose $F$, which is an $\End(C[1] \oplus E)$-valued $A$-form of total degree $2$, as $F = F_{-1} + F_0 + F_1$, where $F_i \in \Omega^{2-i}(A) \otimes \End_i(C[1] \oplus E)$.  Specifically, we have
\begin{align*}
F_{-1} &= D^c \circ \Omega + \Omega \circ D^s, \\
F_0 &= F^c + F^s + \boundary \circ \Omega + \Omega \circ \boundary, \\
F_1 &= \boundary \circ D^c + D^s \circ \boundary.
\end{align*}
It is immediate from (\ref{eqn:curvc}) and (\ref{eqn:curvs}) that $F_0 = 0$.  Similarly, $F_1 = 0$ by (\ref{eqn:relconn}).

To see that $F_{-1} = 0$, we compute the following for $X,Y,Z \in \Gamma(A)$:
\begin{equation*}
\begin{split}
\iota_Z \iota_Y \iota_X D^c \circ \Omega &= \nabla^c_X \Omega_{Y,Z} - \Omega_{[X,Y],Z} + \cyclic\\
&= \nabla^c_X \Omega_{Y,Z} + [\Omega_{X,Y}, \hat{Z}] + \cyclic\\
&= \Omega_{X,Y} \nabla^s_Z + \cyclic\\
&= -\iota_Z \iota_Y \iota_X \Omega \circ D^s.
\end{split}
\end{equation*}
Lemma \ref{lemma:omegacycl} was used in the second line, and (\ref{eqn:relconnphi}) was used in the third line.  We conclude that $\totaldiff^2 =0$, so $\totaldiff$ is flat.
\end{proof}

\subsection{The superconnection in the differential viewpoint} \label{sec:decompdiff}

As we saw in \S\ref{sec:horizontal}, a choice of a horizontal lift $A \to \hat{A}$ is equivalent to a choice of a decomposition $D \iso A \oplus E \oplus C$.  Given such a choice, the space of algebroid cochains may be decomposed as
\begin{equation}\label{eqn:phi}
\bigwedge \Gamma(D\starover{E}, E) \iso \bigwedge \Gamma(A^*) \otimes C^\infty(E) \otimes \bigwedge \Gamma(C^*).
\end{equation}

We restrict our attention to the elements of $\bigwedge \Gamma(D\starover{E}, E)$ that are of degree $1$ with respect to the ``over $A$'' grading of \S\ref{sec:diff}.  Using the decomposition (\ref{eqn:phi}), we may describe the subspace of such elements as
\begin{equation*}
\bigwedge \Gamma(A^*) \otimes \left(C^\infty_\ell(E) \oplus \bigwedge^1 \Gamma(C^*)\right) = 
\Omega(A) \otimes \left(\Gamma(E^*) \oplus \Gamma(C^*[-1])\right).
\end{equation*}
This subspace is invariant under the differential $d_D$, and it is immediate that the restriction of $d_D$ to this subspace is a flat $A$-superconnection on $E^* \oplus C^*[-1]$.

It may be seen that $d_D$ is dual to the superconnection $\totaldiff$ of \S\ref{sec:superconn}, in the sense that, for all $\omega \in \Omega(A) \otimes \left(\Gamma(C[1]) \oplus \Gamma(E)\right)$ and $\eta \in \Omega(A) \otimes \left(\Gamma(E^*) \oplus \Gamma(C^*[-1])\right)$,
\begin{equation*}
\langle \totaldiff\omega, \eta \rangle = d_A \langle \omega,\eta \rangle - (-1)^{|\omega|} \langle \omega, d_D \eta \rangle.
\end{equation*}
From this perspective, we see that the flatness of $\totaldiff$ is equivalent to the fact that $d_D^2 = 0$.

The following theorem, which ties together the main results of \S\ref{sec:genconns}, is an immediate consequence of the above discussion.

\begin{thm}  \label{thm:summary} {\ }
\begin{enumerate}
\item  There is a one-to-one correspondence between $\VB$--algebroid structures on the decomposed DVB $A \oplus E\oplus C$ and flat $A$-superconnections on $C[1] \oplus E$.
\item  Let $D$ be a DVB such as \eqref{lavb}, with side bundles $A$ and $E$, and with core bundle $C$, where $A$ is a Lie algebroid.  After choosing a horizontal lift $\Gamma(A) \to \Gamma(\hat{A})= \Gamma_{\ell}(D,E)$  (or, equivalently, a decomposition $D \iso A \oplus E \oplus C$), there is a one-to-one correspondence between  $\VB$--algebroid structures on  $D$ and flat $A$-superconnections on $C[1] \oplus E$.
\item  A flat $A$-superconnection on $C[1] \oplus E$ is equivalent to
   an $A$--connection $\nabla^c$ on $C$, an $A$--connection $\nabla^s$ on $E$, an operator $\boundary: C \to E$, and an operator $\Omega \in \bigwedge^2\Gamma(A^*) \otimes \Hom(E,C)$, satisfying 
\begin{equation} \label{eqn:conditions}
\begin{split}
\boundary \circ \nabla^c_X = \nabla^s_X \circ \boundary \\
F^c_{X,Y} = \Omega_{X,Y} \circ \boundary \\
F^s_{X,Y} = \boundary \circ \Omega_{X,Y} \\
D^c \Omega + \Omega D^s = 0
\end{split}
\end{equation}
for all $X,Y \in \Gamma(A)$.  Here,  $F^c$ and $F^s$ are the curvatures of $\nabla^c$ and $\nabla^s$; whereas $D^c$ and $D^s$ are the operators on $ \Omega(A) \otimes \Gamma \left( C[1] \oplus E \right )$ associated to $\nabla^c$ and $\nabla^s$.
\end{enumerate}
\end{thm}

In \S\ref{sec:doDiad} we explain how the flat $A$--superconnection depends on the choice of horizontal lift.

\begin{ex}
  A $\VB$--algebroid is said to be \emph{vacant} if the core is trivial.  In this case, there is a unique decomposition $D = A\oplus E$, so by Theorem \ref{thm:summary} there is a one-to-one correspondence between vacant $\VB$--algebroids and Lie algebroid representations.
\end{ex}

\begin{ex}
  In the case where $M$ is a point, so that $A$ is a Lie algebra, it is perhaps surprising that the situation does not simplify much; after choosing a decomposition, we still obtain a flat $A$-superconnection on $C[1] \oplus E$, where $C$ and $E$ are now vector spaces.  In particular, there exist examples that do not correspond to ordinary Lie algebra representations.  This situation is in contrast to that of representations up to homotopy (in the sense of Evens, Lu, Weinstein \cite{elw}), which reduce to ordinary representations when $A$ is a Lie algebra.
\end{ex}

\subsection{Dependence of \texorpdfstring{$\totaldiff$}{D} on the horizontal lift}  \label{sec:doDiad}

As we saw in Theorem \ref{thm:summary}, a $\VB$--algebroid structure on a DVB \eqref{lavb} is, after choosing a horizontal lift, equivalent to a flat $A$-superconnection on $C[1] \oplus E$.  It is then reasonable to wonder how flat $A$-superconnections behave under a change of horizontal lift.  In this section, we will obtain a simple description that may be interpreted as a natural notion of equivalence between two flat $A$-superconnections.  

 The set of horizontal lifts is an affine space modelled on the vector space $\Wii$.  More specifically, consider two horizontal lifts $h, \nue{h} : \Gamma(A) \to \Gamma(\hat{A})$.  For $X \in \Gamma(A)$ we denote $\hat{X} \defequal h(X)$ and $\nue{\hat{X}} \defequal \nue{h}(X)$.  Let $\sigma_X \in \Hom(E,C)$ be defined as
\begin{equation} \label{eq:2hats}
\sigma_X \defequal \nue{\hat{X}} - \hat{X}.
\end{equation}
Equation \eqref{eq:2hats} defines a unique $\sigma \in \Wii = \Omega^1(A) \otimes \Hom(E,C)$.

We may extend $\sigma$ to an operator of total degree $0$ on $\Omega(A) \otimes \Gamma(E \oplus C[1])$, where for $\omega \in \Omega^p(A)$, $\alpha \in \Gamma(C)$, and $\varepsilon \in \Gamma(E)$,
\begin{align*}
\sigma(\omega \alpha) &= 0, && \sigma(\omega \varepsilon) = (-1)^p \omega \wedge \sigma(\varepsilon).
\end{align*}

\begin{thm}\label{thm:vbtoflat}
Let $X \to \hat{X}$ and $X \to \nue{\hat{X}}$ be two horizontal lifts related by $\sigma \in \Wii$ via \eqref{eq:2hats}.  Let $\totaldiff$ and $\nue{\totaldiff}$ be the corresponding superconnections.  
Then
\begin{equation} \label{eq:transftotaldiff}
\nue{\totaldiff} = \totaldiff + [\sigma, \totaldiff] + \frac{1}{2} \left[ \sigma, \left[  \sigma, 
\totaldiff \right]  \right].
\end{equation}
In addition,
\begin{equation*}
\left[ \sigma, \left[ \sigma, \left[ \sigma, \totaldiff \right] \right] \right] = 0.
\end{equation*}
If we denote 
\begin{equation*}
\ad(P_1)P_2 := [P_1,P_2]
\end{equation*}
 for  operators $P_1$ and $P_2$
 on $\Omega(A) \otimes \Gamma(E \oplus C[1])$, then
\eqref{eq:transftotaldiff} can be rewritten as 
\begin{equation*}
\nue{\totaldiff} = \sum_{n=0}^{\infty} \frac{1}{n!} (\ad(\sigma))^n \; \totaldiff  \; = \; \exp(\ad(\sigma)) \; \totaldiff \; = \; u \circ \totaldiff \circ u^{-1}. 
\end{equation*}
In the last equation, $u$ is the automorphism in $\Omega(A) \otimes \Gamma(E \oplus C[1])$ defined by $u = 1 + \sigma$.
\end{thm}

\begin{proof}
Let us write each superconnection as sum of connections and operators, as in \S \ref{sec:superconn}:
\begin{equation*}
\begin{split}
\totaldiff & = D^c + D^s + \boundary + \Omega, \\
\nue{\totaldiff} & = \nue{D^c} + \nue{D^s} + \nue{\boundary} + \nue{\Omega}.
\end{split}
\end{equation*}

A direct calculation from \eqref{eq:2hats} gives us
\begin{align}
\label{eqn:changeDc}
\iota_X \nue{D^c} = & \; \nabla^c_{\nue{\hat{X}}} = \nabla^c_{\hat{X}} + \nabla^c_{\sigma_X} = \iota_X D^c + \sigma_X \circ \boundary,
\\
\label{eqn:changeDs}
\iota_X \nue{D^s} = & \; \nabla^s_{\nue{\hat{X}}} = \nabla^s_{\hat{X}} + \nabla^s_{\sigma_X} = \iota_X D^s + \boundary \circ \sigma_X, 
\\
\label{eqn:changeboundary}
\nue{\boundary} = & \; \boundary,
\\
\nue{\Omega}_{X,Y} = & \; \nue{\hat{[X,Y]}} - [\nue{\hat{X}}, \nue{\hat{Y}}] = 
\hat{[X,Y]} + \sigma_{[X,Y]} - [\hat{X} + \sigma_X, \hat{Y} + \sigma_Y].
\end{align}

According to \eqref{eqn:relconnphi}, we have
\begin{equation*}
\begin{split}
[\sigma_X, \hat{Y}] = \sigma_X \circ \nabla^s_{Y} - \nabla^c_{Y} \circ \sigma_X, \\
[\hat{X}, \sigma_Y] = \nabla^c_{X} \circ \sigma_Y - \sigma_Y \circ \nabla^s_{X},
\end{split}
\end{equation*}
and according to \eqref{eqn:bracketphi}, we have
\begin{equation*}
[\sigma_X, \sigma_Y] = \sigma_X \circ \boundary \circ \sigma_Y - \sigma_Y \circ \boundary \circ \sigma_X,
\end{equation*}
so that
\begin{equation}
\label{eqn:changeOmega}
\begin{split}
\nue{\Omega}_{X,Y} & = 
\Omega_{X,Y} + \sigma_{[X,Y]} - \sigma_X \nabla^s_Y + \sigma_Y \nabla^s_X - \nabla^c_X \sigma_Y + \nabla^c_Y \sigma_X  \\
& \; \; - \sigma_X \boundary \sigma_Y - \sigma_Y \boundary \sigma_X.
\end{split}
\end{equation}

Then we can rewrite \eqref{eqn:changeDc}, \eqref{eqn:changeDs}, \eqref{eqn:changeboundary}, and \eqref{eqn:changeOmega} as
\begin{equation} \label{eqn:changeoper}
\begin{split}
\nue{D^c}  = D^c + \sigma \boundary,  & \quad \quad
\nue{D^s}  = D^s - \boundary \sigma, \\
\nue{\boundary}   = \boundary, & \quad \quad
\nue{\Omega}  = \Omega - D^c \sigma + \sigma D^s - \sigma \boundary \sigma.
\end{split}
\end{equation}

On the other hand, we can write the left-hand side of \eqref{eq:transftotaldiff} in terms of $D^c$, $D^s$, $\boundary$ and $\Omega$ as follows:
\begin{equation} \label{eqn:commDsigma}
\begin{split}
\totaldiff & = D^s + D^c + \boundary + \Omega, \\
[\sigma, \totaldiff] & = \sigma D^s + \sigma \boundary - D^c \sigma - \boundary \sigma, \\
[\sigma, [\sigma, \totaldiff]] & = - 2 \sigma \boundary \sigma, \\
[\sigma, [\sigma, [\sigma, \totaldiff]]] & = 0.
\end{split}
\end{equation}

Finally, comparing \eqref{eqn:changeoper} and \eqref{eqn:commDsigma} completes the proof.
\end{proof}

\section{Characteristic classes}\label{sec:cc}
\label{sec:ccflat}
Given a Lie algebroid $A \to M$ equipped with a representation (i.e.\ a flat $A$-connection) on a vector bundle $E \to M$, Crainic \cite{crainic:vanest} has constructed Chern-Simons-type secondary characteristic classes in $H^{2k-1}(A)$.  In this section we extend his construction to flat $A$-superconnections on graded vector bundles.  In the case of flat $A$-superconnections arising from $\VB$--algebroids, we will see that the associated characteristic classes do not depend on the choice of horizontal lift; in other words, this construction gives us $\VB$--algebroid invariants.

Let $A \to M$ be a Lie algebroid, and let $\Etot = \bigoplus E_i \to M$ be a $\integers$-graded vector bundle\footnote{We assume that the total rank of $\Etot$ is finite, so as to ensure that the supertrace in (\ref{eqn:chernsimons}) is well-defined.} equipped with a flat $A$-superconnection $\totaldiff$.  In other words, $\totaldiff$ is a degree $1$ operator on $\Omega(A) \otimes \Gamma(\Etot)$ satisfying (\ref{eqn:superleibniz}) and such that $\totaldiff^2 = 0$.  Before we can define characteristic classes associated to $\totaldiff$, we will require a few pieces of background.

First, there is a natural pairing
\begin{equation}\label{eqn:pairing}
\Omega(A) \otimes \Gamma(\Etot) \times \Omega(A) \otimes \Gamma(\Etot^*) \to \Omega(A)
\end{equation}
given by $\langle \omega a, \eta \varsigma \rangle = (-1)^{|a||\eta|}\omega \wedge \eta \langle a, \varsigma \rangle$ for all $\omega, \eta \in \Omega(A)$, $a \in \Gamma(\Etot)$, and $\varsigma \in \Gamma(\Etot^*)$.  The adjoint connection $\totaldiff^\dagger$ is an $A$-superconnection on $\Etot^*$ defined by the equation
\begin{equation}\label{eqn:adjoint}
d_A \langle a, \varsigma \rangle = \langle \totaldiff a, \varsigma \rangle + (-1)^{|a|} \langle a, \totaldiff^\dagger \varsigma \rangle.
\end{equation}
It is immediate from the definition that $\totaldiff^2 = 0$ implies that $(\totaldiff^\dagger)^2 = 0$.  

Second, a choice of metric on $E_i$ for all $i$ gives an isomorphism $g: \Etot \iso \Etot^*$, which preserves parity but fails to be degree-preserving; rather, it identifies the degree $i$ component of $\Etot$ with the degree $-i$ component of $\Etot^*$.  Nonetheless, such a choice allows us to transfer $\totaldiff^\dagger$ to a flat $A$-superconnection $\presup{g}\totaldiff$ on $\Etot$.  The superconnection $\presup{g}\totaldiff$ of course depends on $g$, and since $g$ is not degree-preserving, $\presup{g}\totaldiff$ is not homogeneous of degree $1$.  To emphasize this fact, we will refer to $\presup{g}\totaldiff$ as a ``nonhomogeneous superconnection''.

Third, let $I$ be the unit interval, and consider the product Lie algebroid $A \times TI \to M \times I$.  If the canonical coordinates on $T[1]I$ are $\{t,\dot{t}\}$, then any Lie algebroid $p$-form $\mathcal{B} \in \Omega^p(A \times TI)$ may be uniquely expressed as $B_p(t) + \dot{t} B_{p-1}(t)$, where $B_p$ and $B_{p-1}$ are $t$-dependent elements of $\Omega^p(A)$ and $\Omega^{p-1}(A)$, respectively.  Furthermore, in terms of the coordinates on $T[1]I$, the Lie algebroid differential is $d_{A \times TI} = d_A + \dot{t}\pdiff{}{t}$.

Together, $\totaldiff$ and $\presup{g}\totaldiff$ determine an $A \times TI$-(nonhomogeneous) superconnection $\mathcal{T}_{\totaldiff, \presup{g}\totaldiff}$ on $p^*\Etot$, where $p$ is the projection map from $M \times I$ to $M$, such that
\begin{equation}\mathcal{T}_{\totaldiff, \presup{g}\totaldiff}(a) = t\totaldiff(a) + (1-t)(\presup{g}\totaldiff(a)),\label{eqn:transg}
\end{equation}
where $a \in \Gamma(\Etot)$ is viewed as a $t$-independent element of $\Gamma(p^*\Etot)$.  Equation \eqref{eqn:transg}, together with the Leibniz rule \eqref{eqn:superleibniz}, completely determines $\mathcal{T}_{\totaldiff, \presup{g}\totaldiff}$ as an operator on $\Omega(A \times TI) \otimes \Gamma(\Etot)$.

For positive integers $k$, the $k$-th Chern-Simons forms are then
\begin{equation}\label{eqn:chernsimons}
\cs^g_k(\totaldiff) \defequal \int \dee{t} \dee{\dot{t}} \str \left((\mathcal{T}_{\totaldiff, \presup{g}\totaldiff})^{2k}\right).
\end{equation}
The integral in (\ref{eqn:chernsimons}) is a Berezin integral.

For the purpose of clarity, we will spell out what \eqref{eqn:chernsimons} means in more detail.  Since $(\mathcal{T}_{\totaldiff, \presup{g}\totaldiff})^{2k}$ is an even operator on $\Omega(A \times TI) \otimes \Gamma(\Etot)$, its supertrace is an even (in general nonhomogeneous) element of $\Omega(A \times TI)$.  If we express $\str\left((\mathcal{T}_{\totaldiff, \presup{g}\totaldiff})^{2k}\right)$ in the form $B_{\even}(t) + \dot{t}B_{\odd}(t)$, then $\cs^g_k(\totaldiff) = \int_0^1 B_{\odd}(t) dt$.  Therefore $\cs^g_k(\totaldiff) \in \Omega^{\odd}(A)$.

\begin{rmk}
The integral in \eqref{eqn:chernsimons} may be explicitly computed.  The result is that, up to a constant, $\cs^g_k(\totaldiff)$ is given by
\begin{equation*}
\str \left( \totaldiff (\presup{g}\totaldiff \totaldiff)^{k-1} - (\presup{g}\totaldiff \totaldiff)^{k-1} (\presup{g}\totaldiff) \right).
\end{equation*}
\end{rmk}

Since the proofs of the following statements are similar to those of Crainic and Fernandes \cite{cf}, we postpone them to Appendix \ref{appendix:charclasses}.

\begin{prop}\label{prop:csclosed}
For all $k$, $\cs^g_k(\totaldiff)$ is closed.
\end{prop}

\begin{lemma}\label{lemma:keven}
If $k$ is even, then $\cs^g_k(\totaldiff) = 0$.
\end{lemma}

\begin{prop}\label{prop:csodd}
The cohomology class of $\cs^g_k(\totaldiff)$ is an element of $H^{2k-1}(A)$.  In other words, the components of $[\cs^g_k(\totaldiff)]$ in all degrees other than $2k-1$ vanish.
\end{prop}

\begin{prop}\label{prop:metric}
The cohomology class of $\cs^g_k(\totaldiff)$ does not depend on $g$.
\end{prop}

In summary, we have well-defined Chern-Simons classes $[\cs_k(\totaldiff)] \in H^{2k-1} (A)$ associated to any flat $A$-superconnection $\totaldiff$.

Let us now return to $\VB$--algebroids.  We have seen in \S\ref{sec:superconn} that, given a $\VB$--algebroid \eqref{lavb}, a choice of a horizontal lift $A \to \hat{A}$ leads to a flat $A$-superconnection on $C[1] \oplus E$.  Therefore, the above procedure applies, and we may obtain Chern-Simons classes.
\begin{thm}\label{thm:cslift}
The Chern-Simons classes $[\cs_k(\totaldiff)]$ do not depend on the choice of horizontal lift.  Therefore the Chern-Simons classes arising from flat $A$-superconnections on $C[1] \oplus E$ are $\VB$--algebroid invariants.
\end{thm}

\section{Classification of regular \texorpdfstring{$\VB$}{VB}--algebroids}
\label{sec:classification}

Let $D$ be a DVB such as \eqref{lavb}, with side bundles $A$ and $E$, and with core bundle $C$, where $A$ is a Lie algebroid.  In this section we classify the $\VB$--algebroid structures on $D$ that are regular in a sense that will be defined below.

As we saw in Theorem \ref{thm:summary}, given a horizontal lift $A \to \hat{A}$, a $\VB$--algebroid structure on $D$ 
is equivalent to choosing $A$--connections $\nabla^c$ and $\nabla^s$ on $C$ and $E$, respectively, an operator $\boundary: C \to E$, and an operator $\Omega \in \Omega^2(A) \otimes \Gamma(\Hom(E,C))$, satisfying \eqref{eqn:conditions}.  As we saw in \S\ref{sec:doDiad}, only $\boundary$ is intrinsically defined, whereas $\nabla^c$, $\nabla^s$, and $\Omega$ depend on the choice of horizontal lift according to \eqref{eqn:changeoper}.  As a consequence, the set of isomorphism classes of $\VB$--algebroid structures on $D$ is in one-to-one correspondence with the set of tuples $(\nabla^s, \nabla^c, \boundary, \Omega)$ satisfying \eqref{eqn:conditions}, modulo the action of $\Wii$ described by \eqref{eqn:changeoper}.

Nevertheless, as was explained in \S \ref{sec:sideandcore}, when $\boundary$ is of constant rank, $\nabla^c$ and $\nabla^s$ induce the following two $A$--connections that depend only on the total $\VB$--algebroid structure, and not on the choice of horizontal lift:  
\begin{itemize}
\item a flat $A$--connection $\nabla^K$ on the subbundle $K := \ker \boundary \subseteq C$, and
\item a flat $A$--connection $\nabla^{\nu}$ on the quotient bundle $\nu := \coker \boundary = E / \im \boundary$.
\end{itemize}

The elements $(A, E, C, \boundary, \nabla^K, \nabla^{\nu})$ are all invariant under isomorphisms of $\VB$--algebroids.  We will see that any such $6$-tuple can always be ``extended'' to a $\VB$--algebroid, and we will classify the extensions up to isomorphism.

\begin{dfn}
A $\VB$--algebroid is called \emph{regular} when the core-anchor $\boundary:C \to E$ has constant rank. 
\end{dfn}

Note that in a regular $\VB$--algebroid, the Lie algebroids $D \to E$ and $A \to M$ do not have to be regular (i.e.\ the anchor maps do not have to have constant rank). For example, a $\VB$--algebroid where $\boundary$ is an isomorphism is clearly regular; however, in such a $\VB$--algebroid the anchors $\rho_D$ and $\rho_A$ need not be of constant rank.  This fact will be more clearly illustrated in \S\ref{sec:type1}.

There are two special types of regular $\VB$--algebroids.  We will describe them now, and then we will show that any regular $\VB$--algebroid can be uniquely decomposed as a direct sum of these two special types of $\VB$--algebroids.  This will allow us to give a complete description of all regular $\VB$--algebroids up to isomorphism.

\subsection{\texorpdfstring{$\VB$}{VB}--algebroids of type 1} \label{sec:type1}

\begin{dfn} We say that a $\VB$--algebroid is \emph{of type 1} when the core-anchor $\boundary$ is an isomorphism of vector bundles.
\end{dfn}

There is one canonical example (which turns out to be the only one).  Let $A \to M$ be a Lie algebroid and $E \to M$ be a vector bundle.  Consider the pullback of $TE$ by the anchor $\rho_A$ of $A$ in the following diagram:
\begin{equation}
\xymatrix{\rho_A^*(TE) \ar[r] \ar[d] & TE \ar[d] \\ A \ar[r]^{\rho_A} & TM}
\end{equation}
Then there is a natural pullback Lie algebroid structure (see \cite{higgins-mac}) on $\rho_A^*(TE) \to E$ such that
\begin{equation}\label{eqn:rhostar}
\xymatrix{\rho_A^*(TE) \ar[r] \ar[d] & E \ar[d] \\ A \ar[r] & M}
\end{equation}
is a $\VB$--algebroid of type 1.  The core of \eqref{eqn:rhostar} may be canonically identified with $E$, and the core-anchor map\footnote{Note that a minus sign already appears in the definition of the core-anchor (Definition \ref{dfn:coreanchor}).} is $-\id_E$.

Let us try to construct the most general $\VB$--algebroid of type 1.  Let us fix the sides $A$ and $E$.  We may assume that $C=E$ and $\boundary = -1$.  Now we need to define $\nabla^s$, $\nabla^c$, and $\Omega$ satisfying \eqref{eqn:conditions}.  In this case, the equations become:
\begin{itemize}
\item $\nabla^s = \nabla^c$,
\item and -$\Omega$ is the curvature of $\nabla^s$.
\end{itemize}

Hence, putting a $\VB$--algebroid structure on $A \oplus E \oplus E$ is the same thing as defining an $A$--connection on $E$.  If we want to classify them up to isomorphism we need to include the action of $\Gamma(A^* \otimes E^* \otimes E) = \Omega^1(A) \otimes \End(E)$ by \eqref{eqn:changeoper}.  Given any two $A$--connections $\nabla$ and $\nue{\nabla}$ on $E$ there exist a unique $\sigma \in \Omega^1(A) \otimes \End(E)$ such that
$\nue{\nabla}=\nabla + \sigma$. 
In other words:

\begin{prop}
Given side bundles $A$ and $E$, there exists a unique $\VB$--algebroid of type 1 up to isomorphism, namely $\rho_A^*(TE)$.
\end{prop}

\subsection{\texorpdfstring{$\VB$}{VB}--algebroids of type 0} \label{sec:type0}

\begin{dfn} We say that a $\VB$--algebroid is \emph{of type 0} when the core anchor is zero.
\end{dfn}

Fix the sides $A$ and $E$, and the core $C$.  Let us try to construct the most general $\VB$--algebroid with $\boundary = 0$.
We need to define $\nabla^s$, $\nabla^c$, and $\Omega$ satisfying \eqref{eqn:conditions}.  In this case, the equations become:
\begin{itemize}
\item $\nabla^s$ is a flat $A$--connection on $E$,
\item $\nabla^c$ is a flat $A$--connection on $C$,
  \item $\displaystyle{D^c \circ \Omega + \Omega \circ D^s=0}$ \hfill \refstepcounter{equation}{\rm (\theequation)} \label{eqn:dOmega}
\end{itemize}

To classify these $\VB$--algebroids up to isomorphism we need to include the action of $\Wii$ by \eqref{eqn:changeoper}.    In this case, if $\sigma \in \Wii$ acts on $(\nabla^s,\nabla^c, \Omega)$, the connnections $\nabla^s$ and $\nabla^c$ remain invariant, whereas $\Omega$ becomes:
\begin{equation} \label{eqn:dsigma}
\nue{\Omega} = \Omega + \sigma D^s - D^c \sigma
\end{equation}

Equations \eqref{eqn:dOmega} and \eqref{eqn:dsigma} can be interpreted in terms of cohomology.  Namely, the flat $A$-connections on $C$ and $E$ induce a flat $A$-connection on $\Hom(E,C)$, whose covariant derivative $D$ is given by the equation
\begin{equation}
D \alpha : = \alpha D^s + (-1)^p D^c \alpha
\end{equation}
for $\alpha \in \Omega^p(A) \otimes \Gamma( \Hom(E,C))$.  Then \eqref{eqn:dOmega} says that $D \Omega = 0$, whereas \eqref{eqn:dsigma} says that $\nue{\Omega} = \Omega + D\sigma$.  Hence, the cohomology class $\left[ \Omega \right] \in H^2(A; \Hom(E,C))$ is well-defined and invariant up to isomorphism of $\VB$--algebroids.  This gives us the following result:

\begin{prop}
Type $0$ $\VB$--algebroids with sides $A$ and $E$, and core $C$ are classified up to isomorphism by triples $(\nabla^s, \nabla^c, [\Omega])$, where
\begin{itemize}
\item $\nabla^s$ is a flat $A$--connection on $E$,
\item $\nabla^c$ is a flat $A$--connection on $C$,
\item $[\Omega]$ is a cohomology class in $H^2(A; \Hom(E,C))$.
\end{itemize}
\end{prop}

\subsection{The general case}

Given two $\VB$--algebroids
\begin{equation*}
\xymatrix{D_1 \ar[r] \ar[d] & E_1 \ar[d] \\ A \ar[r] & M}
\quad \quad
\xymatrix{D_2 \ar[r] \ar[d] & E_2 \ar[d] \\ A \ar[r] & M}
\end{equation*}
over the same Lie algebroid $A$, we can obtain the \emph{direct sum} $\VB$--algebroid
\begin{equation*}
\xymatrix{D_1 \oplus_A D_2 \ar[r] \ar[d] & E_1 \oplus_M E_2 \ar[d] \\ A \ar[r] & M}.
\end{equation*}
Note that the core of $D_1 \oplus_A D_2$ is the direct sum of the cores of $D_1$ and $D_2$.

\begin{thm} \label{thm:0and1}
Given a regular $\VB$--algebroid $D$, there exist unique (up to isomorphism) $\VB$--algebroids $D_0$ of type 0, and $D_1$ of type 1, such that $D$ is isomorphic to $D_0 \oplus_A D_1$.
\end{thm}
\begin{proof} \phantom{1} 

$\bullet$ {\bf Existence.} \\
Let $D$ be a regular $\VB$--algebroid as in \eqref{lavb}.  Then the core-anchor $\boundary: C \to E$ induces the following vector bundles:
$K := \ker \boundary \subseteq C$,  
$F:= \im \boundary \subseteq E$, and
$\nu := \coker \boundary = E / F$.  They fit into the short exact sequences:
\begin{equation}\label{eqn:knushort}
\begin{split}
    &\xymatrix{K \ar[r] & C \ar[r] & F}, \\
&\xymatrix{F \ar[r] & E \ar[r] & \nu}.
\end{split}
\end{equation}
Let us choose splittings of the sequences \eqref{eqn:knushort}, which would give isomorphisms
\begin{equation} \label{eqn:1TI}
\begin{split} 
C & \approx K \oplus F \\
E & \approx \nu \oplus F 
\end{split}
\end{equation}

Next we make a choice of horizontal lift $X \in \Gamma(A) \to \hat{X} \in \Gamma(\hat{A})$.  As we saw in Theorem \ref{thm:summary}, the $\VB$--algebroid structure in the DVB $D$ is determined by the tuple $(\nabla^s, \nabla^c, \boundary, \Omega)$.  We write a ``block-matrix decomposition'' of each one of these operators with respect to the direct sums in \eqref{eqn:1TI}:
\begin{equation} \label{eqn:blockmatrix}
\nabla^s = \begin{pmatrix} \nabla^{\nu} & 0 \\ \Lambda  & \nabla^F \end{pmatrix}, \quad
\nabla^c = \begin{pmatrix} \nabla^K & \Gamma \\ 0 & \nabla^F \end{pmatrix}, \quad
\boundary = \begin{pmatrix} 0 & 0 \\ 0 & -1 \end{pmatrix}, \quad
\Omega = \begin{pmatrix} \alpha & \star  \\ \star & \star \end{pmatrix}. 
\end{equation}
In \eqref{eqn:blockmatrix}, $\star$ means an unspecified operator.  The zeros in $\nabla^s$ and $\nabla^c$ are a consequence of the first equation in \eqref{eqn:conditions}.  The bottom--right blocks of $\nabla^s$ and $\nabla^c$ (which we denote $\nabla^F$) are the same, also because of the first equation in \eqref{eqn:conditions}.  The components $\alpha$, $\Lambda$, and $\Gamma$ are described as follows:
\begin{equation*}
\begin{split}
\alpha & \in    \Lambda^2 \Gamma(A) \otimes \Gamma(\nu) \to \Gamma(K) \\
\Lambda & \in  \Gamma(A) \otimes \Gamma(\nu) \to \Gamma(F) \\
\Gamma & \in   \Gamma(A) \otimes \Gamma(F) \to \Gamma(K)
\end{split}
\end{equation*}
The operators $\alpha$, $\Lambda$, and $\Gamma$ depend on the choice of splittings of \eqref{eqn:knushort}, as well as on the choice of horizontal lift.  The $A$-connection $\nabla^F$ depends on the choice of horizontal lift.

If the operators in \eqref{eqn:blockmatrix} were block-diagonal, then we could break them apart to form two separate $\VB$-algebroid structures, one with side bundle $K$ and core $\nu$, and the other with $F$ as both the side and core.  Luckily, it is possible to make all the operators in \eqref{eqn:blockmatrix} block-diagonal via a change of horizontal lift, as follows.  

As we explained in \textsection \ref{sec:doDiad}, a change of horizontal lift corresponds to an element $\sigma \in \Wii$.  If $\sigma$ is written in block matrix form as
\begin{equation} \label{eqn:gensigma}
\sigma = \begin{pmatrix} \sigma^{11} & \sigma^{12} \\ \sigma^{21} & \sigma^{22} \end{pmatrix},
\end{equation}
then, according to \eqref{eqn:changeDc} and \eqref{eqn:changeDs}, the side and core connections for the new horizontal lift will be
\begin{align*}
\nue{\nabla^s} &= \nabla^s + \boundary \sigma =
\begin{pmatrix} \nabla^{\nu} & 0 \\ \Lambda & \nabla^F \end{pmatrix} +
\begin{pmatrix} 0 & 0 \\ -\sigma^{21} & -\sigma^{22} \end{pmatrix}, 
\\
\nue{\nabla^c} & = \nabla^c + \sigma \boundary  =
\begin{pmatrix} \nabla^K & \Gamma \\ 0 & \nabla^F \end{pmatrix} +
\begin{pmatrix} 0 & -\sigma^{12} \\ 0 & -\sigma^{22} \end{pmatrix}.
\end{align*}
Therefore, if we choose 
\begin{equation} \label{eqn:antidiagsigma}
\sigma = \begin{pmatrix} 0 & \Gamma \\ \Lambda & 0 \end{pmatrix},
\end{equation}
it will make the new connections $\nue{\nabla^s}$ and $\nue{\nabla^c}$ block--diagonal.  Consequently, the second and third equations in \eqref{eqn:conditions} imply that $\nue{\Omega}$ will also be block--diagonal.  In particular, $\nue{\Omega}$ will necessarily take the form
\begin{equation} \label{eqn:blockdiagonal}
\nue{\Omega} = \begin{pmatrix} \omega & 0  \\ 0 & -R^F \end{pmatrix},
\end{equation}
where $R^F$ is the curvature of $\nabla^F$.  Using \eqref{eqn:antidiagsigma} in \eqref{eqn:changeOmega}, we can relate the upper-left block $\omega$ of $\nue{\Omega}$ to the upper-left block $\alpha$ of $\Omega$ in the following way:
\begin{equation} \label{eqn:defomega}
\omega_{X,Y} = \alpha_{X,Y} - \Gamma_X \circ \Lambda_Y + \Gamma_Y \circ \Lambda_X.
\end{equation}

Since $\nue{\nabla^s}$, $\nue{\nabla^c}$, and $\nue{\Omega}$ are block diagonal, their diagonal blocks give us the data for two $\VB$-algebroids: a type 0 $\VB$-algebroid $D_0$, with side bundle $\nu$ and core bundle $K$, and a type 1 $\VB$-algebroid $D_1$, with $F$ as both side and core.

$\bullet$  {\bf Uniqueness.} \\
Based on the classification of $\VB$-algebroids of type 0 (\S\ref{sec:type0}) and type 1 (\S\ref{sec:type1}), we may characterize the $\VB$-algebroids $D_0$ and $D_1$ up to isomorphism as follows:
\begin{itemize}
\item $D_1$ is determined up to isomorphism by its side bundle $F$,
\item $D_0$ is determined up to isomorphism by its side bundle $\nu$ and core bundle $K$, the flat $A$--connections $\nabla^{\nu}$ and $\nabla^K$, and the cohomology class $[\omega] \in H^2(A;\Hom(\nu,K))$, given by \eqref{eqn:defomega}.
\end{itemize}
We have already seen that the bundles $F$, $\nu$, $K$, and the flat $A$--connections $\nabla^{\nu}$ and $\nabla^K$ are canonical.  To complete the proof we need to show that the cohomology class of $\omega$ does not depend on the choice of splittings of \eqref{eqn:knushort}, nor on the choice of horizontal lift.

First, the cohomology class does not depend on the choice of horizontal lift, thanks to our analysis of type 0 and type 1 $\VB$--algebroids. 
If we fix the choice of complements but change to a different horizontal lift that still makes the operators in \eqref{eqn:blockdiagonal} block-diagonal, this corresponds to choosing arbitrary blocks in the main diagonal of \eqref{eqn:gensigma}.  Notice that the cohomology class of $\omega$ does not change, and in fact all the representatives of the cohomology class of $\omega$ may be obtained in this way.

Second, suppose that we have chosen splittings of the sequences \eqref{eqn:knushort} and a horizontal lift such that the operators are already block--diagonal like in \eqref{eqn:blockdiagonal}.  Then a change of splitting of the second sequence in \eqref{eqn:knushort} may be expressed in block form by a matrix
\begin{equation*}
\begin{pmatrix} 1 & 0 \\ g & 1 \end{pmatrix} 
\end{equation*}
for some linear map $g: \nu \to F$.  Under the change of splitting, $\nabla^c$ and $\boundary$ will have the same matrix forms, whereas the new block matrix forms for $\nabla^s$ and $\Omega$ will be
\begin{equation*}
\nabla^s = 
\begin{pmatrix} 1 & 0 \\ -g & 1 \end{pmatrix}
\begin{pmatrix} \nabla^{\nu} & 0 \\ 0 & \nabla^F \end{pmatrix}
\begin{pmatrix} 1 & 0 \\ g & 1 \end{pmatrix} =
\begin{pmatrix} \nabla^{\nu} & 0 \\ -g \nabla^{\nu} + \nabla^F g & 1 \end{pmatrix}
\end{equation*}
and $\Omega$ will be:
\begin{equation*}
\Omega =
\begin{pmatrix} \omega & 0 \\ 0 & -R^F \end{pmatrix}
\begin{pmatrix} 1 & 0 \\ g & 1 \end{pmatrix} =
\begin{pmatrix} \omega & 0 \\ \omega - R^F g & -R^F \end{pmatrix}
\end{equation*}

It is clear that $\omega$, as defined by \eqref{eqn:defomega}, stays the same.  A similar calculation shows that $\omega$ does not depend on the choice of splitting of the first sequence in \eqref{eqn:knushort}.

Notice that as a consequence of the above analysis, if we were to start with an arbitrary choice of splittings of \eqref{eqn:knushort} \emph{and} an arbitrary choice of horizontal lift, then $\omega$ may change under a change of splitting, but only by an exact term.  This can alternatively be shown by a direct (and lengthy) calculation.
\end{proof}

\begin{cor} [Classification of regular $\VB$--algebroids] \label{cor:classification}
A regular $\VB$--algebroid is described, up to isomorphism, by a unique tuple 
$(M,A,E,C,\boundary,\nabla^K, \nabla^{\nu}, [\omega ])$, where
\begin{itemize}
\item $M$ is a manifold,
\item $A \to M$ is a Lie algebroid,
\item $E \to M$ and $C \to M$ are vector bundles,
\item $\boundary: C \to E$ is a morphism of vector bundles,
\item $\nabla^K$ is a flat $A$--connection on $K \defequal \Ker \boundary$,
\item $\nabla^{\nu}$ is a flat $A$--connection on $\nu \defequal \coKer \boundary$,
\item $[\omega ]$ is a cohomology class in $H^2(A; \Hom(\nu,K))$.
\end{itemize}
\end{cor}

\section{Example: \texorpdfstring{$TA$}{TA}}\label{sec:ta}

Let $A \to M$ be a Lie algebroid.  If $A$ is a regular Lie algebroid, i.e.\ if the anchor map $\rho_A: A \to TM$ is of constant rank, then the $\VB$--algebroid $TA$ in \eqref{ex:ta} is regular.  Then, by Corollary \ref{cor:classification}, there is an associated cohomology class $[\omega] \in H^2(A; \Hom(\nu,K))$, where $K$ and $\nu$ are the kernel and cokernel of $\rho_A$, respectively.  Since the construction of the $\VB$--algebroid $TA$ from $A$ is functorial, the class $[\omega]$ is a characteristic class of $A$.  In this section, we will give a geometric interpretation of $[\omega]$ in this case.

As was noted in Example \ref{ex:fatjet}, the fat algebroid $\hat{A}$ in this case is simply the first jet bundle $J^1 A$ of $A$.  There is a natural map $j: \Gamma(A) \to \Gamma(J^1 A)$, which however is not $C^\infty(M)$-linear; instead, it satisfies the property
\begin{equation*}
j(fX) = fj(X) + df \cdot X
\end{equation*}
for $f \in C^\infty(M)$ and $X \in \Gamma(A)$.  Here, $df \cdot X \in \Hom(TM,A)$ is viewed as a jet along the zero section of $A$.

If we choose a linear connection $\tilde{\nabla} : \vect(M) \times \Gamma(A) \to \Gamma(A)$, we may obtain a horizontal lift $X \in A \mapsto \hat{X} \in J^1 A$, where $\hat{X} \defequal j(X) - \tilde{\nabla} X$.  The resulting side and core connections are described as follows:
\begin{align}
\nabla^c_X Y &= [X,Y]_A + \tilde{\nabla}_{\rho_A(Y)} X, \label{eqn:nablacore}\\
\nabla^s_X \phi &= [\rho_A(X),\phi] + \rho_A \left( \tilde{\nabla}_{\phi} X \right),\label{eqn:nablaside}
\end{align}
for $X,Y \in \Gamma(A)$ and $\phi \in \Gamma(TM)$.

Additionally, one can derive the following expression for $\Omega \in \Omega^2(A)\otimes \Gamma(\Hom(TM,A))$:
\begin{equation}\label{eqn:omegata}
\Omega_{X,Y} \phi = [\tilde{\nabla}_\phi X, Y] + [X, \tilde{\nabla}_\phi Y] - \tilde{\nabla}_\phi [X, Y] - \tilde{\nabla}_{\nabla^s_X \phi} Y + \tilde{\nabla}_{\nabla^s_Y \phi} X.
\end{equation}

\subsection{The case \texorpdfstring{$\rho = 0$}{of trivial anchor}}
It is perhaps instructive to begin with the case where the anchor map $\rho_A$ is trivial (or in other words, where $A$ is simply a bundle of Lie algebras).  Since for the $\VB$--algebroid $TA$ in \eqref{ex:ta} we have $\boundary = -\rho_A$, the case $\rho_A = 0$ corresponds to the case where $TA$ is a $\VB$--algebroid of type $0$ (see \S\ref{sec:type0}).

The vanishing of the cohomology class $[\Omega]$ is equivalent to the existence of a connection $\tilde{\nabla}$ for which $\Omega$, described by \eqref{eqn:omegata}, vanishes.

Since $\nabla^s$ becomes trivial when $\rho_A = 0$, we immediately see that $\Omega \phi$ measures the failure of $\tilde{\nabla}_\phi$ to be a derivation of the Lie bracket.  Therefore, $\Omega = 0$ precisely when, for any $\phi \in \vect(M)$, parallel transport along $\phi$ induces Lie algebra isomorphisms of the fibres of $A$.  In fact, it can be shown that if $\Omega = 0$, one can use parallel transport to locally trivialize $A$ as a Lie algebra bundle.  Conversely, given a local trivialization of $A$, one can define parallel transport in a way that respects the Lie brackets on the fibres of $A$.  Thus we have the following result:
\begin{prop}
Let $A \to M$ be a Lie algebroid with $\rho_A = 0$, and consider the type $0$ $\VB$--algebroid $TA$ in \eqref{ex:ta}.  The cohomology class $[\Omega] \in H^2(A; \Hom(TM, A))$ vanishes if and only if the bundle of Lie algebras $A$ is locally trivializable as a Lie algebra bundle.
\end{prop}

\subsection{The general case}

Now we will consider the general case of a regular Lie algebroid $A \to M$.  Let $K \subseteq A$ be the kernel of $\rho_A$, and let $F \subseteq TM$ be the image of $\rho_A$.  The vanishing of the cohomology class $[\omega]$ is equivalent to the existence of a connection $\tilde{\nabla}$ and splittings $A \cong K \oplus F$ and $TM \cong \nu \oplus F$ such that $\omega$, defined in \eqref{eqn:defomega}, vanishes.

First, if we choose a splitting of the short exact sequence of vector bundles
  $K \to A \to F$,
then we obtain an $F$-connection $\nabla^K$ on $K$ and a $K$-valued $2$-form $B \in \Omega^2(F) \otimes \Gamma(K)$, defined by the properties
\begin{equation}  \label{eqn:brackfk}
[\phi, k]_A = \nabla^K_\phi k,  \quad \quad
[\phi, \phi']_A = B(\phi, \phi') + [\phi, \phi']_{TM} 
\end{equation}
for $\phi, \phi' \in \Gamma(F)$ and $k \in \Gamma(K)$.  Note that this $\nabla^K$ is not the same as the one in \eqref{eqn:blockmatrix}.  Choose an extension of $\nabla^K$ to a $TM$-connection $\tilde{\nabla}^K$ on $K$.  

Second, choose a splitting of the sequence $F \to TM \to \nu$.  This induces a $\nu$-connection $\nabla^F$ on $F$, where $\nabla^F_\psi \phi$ is the component of $[\psi, \phi]$ in $F$, for $\psi \in \Gamma(\nu)$ and $\phi \in \Gamma(F)$.  Note that this $\nabla^F$ is not the same as the one in \eqref{eqn:blockmatrix}.  Choose an extention of $\nabla^F$ to a $TM$-connection $\tilde{\nabla}^F$ on $F$.

Third, we may define a $TM$-connection $\tilde{\nabla}$ on $A$ as follows:
\begin{equation}\label{eqn:nablata}
\tilde{\nabla}_\psi X = \tilde{\nabla}^K_\psi X_K + \tilde{\nabla}^F_\psi X_F + B(\psi_F, X_F)
\end{equation}
for $\psi \in \vect(M)$ and $X \in \Gamma(A)$.  Here, $X_K$ and $X_F$ are the components of $X$ in $K$ and $F$, respectively, and $\psi_F$ is the component of $\psi$ in $F$.  

We have constructed the connection $\tilde{\nabla}$ in \eqref{eqn:nablata} such that it has the following properties: 
\begin{itemize}
  \item If $Y_K = 0$, then $\nabla^c_X Y$ is in $F$. We see this by substituting \eqref{eqn:nablata} into \eqref{eqn:nablacore}.
\item If $\psi \in \Gamma(\nu)$, then $\nabla^s_X \psi$ is also in $\Gamma(\nu)$. We see this by substituting \eqref{eqn:nablata} into \eqref{eqn:nablaside}.
\end{itemize}
In other words, if the core and side connections are expressed in block form as in \eqref{eqn:blockmatrix}, then they will both be block-diagonal, and as a consequence, $\Omega$ will also be block-diagonal.  Therefore, to compute $\omega$, we simply need to restrict $\Omega$ to $\Gamma(\nu)$, and the result lies in $K$.

Using \eqref{eqn:brackfk} and \eqref{eqn:nablata} in \eqref{eqn:omegata}, we obtain from a long but direct computation the following equation for $X,Y \in \Gamma(A)$ and $\psi \in \Gamma(\nu)$:
\begin{equation}\label{eqn:omegata2}
\begin{split}
\Omega_{X,Y}\psi =& [\tilde{\nabla}^K_\psi X_K, Y_K]_K + [X_K, \tilde{\nabla}^K_\psi Y_K]_K - \tilde{\nabla}^K_\psi [X_K, Y_K]_K \\
&+ \tilde{R}^K_{X_F, \psi} Y_K - \tilde{R}^K_{Y_F, \psi} X_K \\
&+ \tilde{\nabla}^K_\psi B(X_F, Y_F) - B(\tilde{\nabla}^F_\psi X_F, Y_F) - B(X_F, \tilde{\nabla}^F_\psi Y_F).
\end{split}
\end{equation}
Here, $\tilde{R}^K$ is the curvature of $\tilde{\nabla}^K$. 
\begin{prop}\label{prop:omegata}
$\omega_{X,Y} \psi$ vanishes for all $X,Y \in \Gamma(A)$ and $\psi \in \Gamma(\nu)$ if and only if the following statements are true:
\begin{enumerate}
\item $\tilde{\nabla}^K$ is a derivation of the bracket on $K$,
\item $\tilde{R}^K_{\phi, \psi}$ vanishes for all $\phi \in \Gamma(F)$ and $\psi \in \Gamma(\nu)$, and
\item $B(\tilde{\nabla}^F_\psi \phi, \phi') + B(\phi, \tilde{\nabla}^F_\psi \phi') - \tilde{\nabla}^K_\psi B(\phi, \phi')$ vanishes for all $\phi, \phi' \in \Gamma(F)$ and $\psi \in \Gamma(\nu)$.
\end{enumerate}
\end{prop}
\begin{proof}
First, notice that the restrictions we have imposed in the choice of $TM$--connection $\tilde{\nabla}$ on $A$ in the above construction are equivalent to asking that the operators in \eqref{eqn:blockdiagonal} are block-diagonal.  As was mentioned in the uniqueness part of the proof of Theorem \ref{thm:0and1}, with this restriction to the choices we still get all the forms in the cohomology class $[\omega]$.  Hence,  $[\omega]=0$ if and only if there is a choice of complements and linear connection as the ones above for which $\omega=0$.

Second,  by alternatively setting $X_F, Y_F = 0$, $X_K, Y_F = 0$, and $X_K, Y_K = 0$ in \eqref{eqn:omegata2}, we obtain the required result.
\end{proof}

Given a leaf $L$ of the foliation $F$, the structure of the restricted Lie algebroid $A|_L$ is completely determined by the data $\left([\cdot,\cdot]_K, \nabla^K, B\right)$ over $L$.  Thus, we may interpret the three conditions in Proposition \ref{prop:omegata} as saying that, if $\psi$ respects the foliation, then parallel transport along $\psi$ gives isomorphisms of the restrictions of $A$ to the leaves.  In other words, $[\omega]$ is the obstruction to local trivializability of $A$, in the following sense:

\begin{thm}
The cohomology class $[\omega] \in H^2(A; \Hom(\nu, K))$ vanishes if and only if, around any leaf $L$, there locally exists a tubular neighborhood $\tilde{L}$ and an identification $\tilde{L} \equiv L \times U$ such that the Lie algebroid $A|_{\tilde{L}}$ is isomorphic to the cross product of $A|_L$ and the trivial Lie algebroid over $U$.
\end{thm}

Clearly, the vanishing of $[\omega]$ imposes a strong regularity condition on the Lie algebroid structure of $A$.  In general, we may view $[\omega]$ as a measure of how the Lie algebroid structure on $A|_L$ depends on the choice of $L$.

\appendix

\section{Proof of Theorem \ref{thm:lavbvbla}}
\label{app:sth}

Let us first concentrate on the aspects of the compatibility conditions that relate to the anchor map $\rho_D$.  For $\VB$--algebroids, the requirement is that $\rho_D$ be a bundle morphism as in (\ref{anchordiag}).  For $\LA$--vector bundles, we require that the diagram (\ref{anchordiag}), as well as the diagram
\begin{equation}\label{anchor}
\xymatrix{D^{(2)} \ar^{\rho_D^{(2)}}[r] \ar_{+_A}[d] & TE^{(2)} \ar^{T(+)}[d] \\ D \ar^{\rho_D}[r] & TE},
\end{equation}
commute.  Here, $D^{(2)} \defequal D \times_A D$ and $TE^{(2)} \defequal TE \times_{TM} TE$.  It is immediately clear that the $\VB$--algebroid and $\LA$--vector bundle compatibility conditions for $\rho_D$ are equivalent to each other.  In what follows, we will assume that they are satisfied.

We now turn to the aspects of the compatibility conditions that involve the brackets.  For $\VB$--algebroids, these are conditions (1)-(3) in Definition \ref{dfn:vbla}.  For $\LA$--vector bundles, we require that the Lie algebroid structure on $D \to E$ be $q$--projectible (as defined in Remark \ref{rem:proj}), and that the map $+_A : D \times_A D \to D$ be an algebroid morphism.  We note that Definition \ref{dfn:vbla} refers only to brackets of linear and core sections.  In order to prove the equivalence of the $\VB$--algebroid and $\LA$--vector bundle compatibility conditions, we will rewrite the latter in terms of linear and core sections.

First, let us consider the condition that the Lie algebroid structure on $D \to E$ is to be $q$-projectible.  
\begin{lemma}\label{lemma:qbracket}
The algebroid structure on $D \to E$ is $q$-projectible if and only if, for all $X,Y \in \Gamma_\ell(D,E)$ and $\alpha, \beta \in \Gamma_C(D,E)$, 
\begin{enumerate}
\item $[X,Y]_D$ is $q$-projectible,
\item $[X,\alpha]_D$ is $q$-projectible to $0^A$,
\item $[\alpha,\beta]_D$ is $q$-projectible to $0^A$.
\end{enumerate}
\end{lemma}
\begin{proof}

As in Remark \ref{rmk:coords}, let us pick a decomposition $D \iso A \oplus E \oplus C$ and choose local coordinates  $\{x^i, a^i, e^i, c^i\}$ on $D$, where $\{x^i\}$ are coordinates on $M$, and $\{a^i\}$, $\{e^i\}$, and $\{c^i\}$ are fibre coordinates on $A$, $E$, and $C$, respectively.  Let $\{A_i, C_i\}$ be the frame of sections dual to the fibre coordinates $\{ a^i, c^i\}$.  In Remark \ref{rmk:coords} we described the form of linear and core sections in these coordinates. We now notice that  a section $X \in \Gamma(D,E)$ is $q$--projectible to a section $X_0 = f^i(x)A_i \in \Gamma(A,M)$  
 if and only if it is of the form
\begin{equation*}
X = f^i(x)A_i + g^i(x,e)C_i
\end{equation*}
This coordinate description shows that the space of $q$--projectible sections of $D \to E$ is exactly
\begin{equation} \label{eqn:qproj}
\Gamma_l (D,E) \; + \; \mathcal{C}^{\infty}(E) \otimes \Gamma_c (D,E)
\end{equation}
In terms of the brackets, $q$-projectibility is equivalent to the following two properties:
\begin{itemize}
 \item If $X$ and $Y$ in $\Gamma(D,E)$ are $q$-projectible, then $[X,Y]_D$ is $q$-projectible.
\item If $\alpha \in \Gamma(D,E)$ is $q$-projectible to $0^A$ and $X \in \Gamma(D,E)$ is $q$-projectible, then $[X, \alpha]_D$ is $q$-projectible to $0^A$.
\end{itemize}
From \eqref{eqn:qproj} we can see that these two properties are satisfied if and only if they are satisfied for linear and core sections.  Conditions (1)--(3) in the statement of this lemma are exactly these two properties restricted to linear and core sections.
\end{proof}

Second, we want to transform the condition that $+_A$ be a Lie algebroid morphism into a condition involving only linear and core sections.  In order to do so, we need some definitions.

Let $\tilde{X}$ be a section of $D^{(2)}$ over $E^{(2)} \defequal E \times_M E$.  We say that $\tilde{X}$ is \emph{$+$-projectible} to  $X \in \Gamma(D,E)$ if $+_A \circ \tilde{X} = X \circ +$, i.e.\ if it is a ``q--projectible'' section of the DVB
\begin{equation*}
\xymatrix{ D^{(2)} \ar[r]  \ar[d]_{+_A} & E^{(2)} \ar[d]_{+} \\
D \ar[r] & E }.
\end{equation*}

  If two sections $X, X' \in \Gamma(D,E)$ are both $q$-projectible to the same $X_0 \in \Gamma(A)$, then we may form the product $X \times X' \in \Gamma(D^{(2)}, E^{(2)})$.  
In particular, any $q$-projectible section $X \in \Gamma(D,E)$ induces the lift $X^{(2)} \defequal X \times X$.  
In addition, given any section $\alpha \in \Gamma(D,E)$ that is $q$--projectible to $0^A$, we can define sections 
\begin{equation*}
 \alpha^+ \defequal \frac 12 ( \alpha \times \tilde{0}^E + \tilde{0}^E \times \alpha) \quad \quad \textrm{and} \quad \quad
 \alpha^- \defequal \frac 12  ( \alpha \times \tilde{0}^E - \tilde{0}^E \times \alpha)
\end{equation*}
Let us introduce the following notation, just for the next lemma:
\begin{itemize}
\item $\Gamma_l^{(2)}$ denotes the set of lifts of sections $X \in \Gamma_l(D,E)$ to sections $X^{(2)} \in \Gamma (D^{(2)},E^{(2)})$,
\item $\Gamma_c^+$ denotes the set of lifts of sections $\alpha \in \Gamma_c(D,E)$ to sections $\alpha^+ \in \Gamma (D^{(2)},E^{(2)})$,
\item $\Gamma_c^-$ denotes the set of lifts of sections $\alpha \in \Gamma_c(D,E)$ to sections $\alpha^- \in \Gamma (D^{(2)},E^{(2)})$,
\item $\mathcal{C}^{\infty}(E)^{(2)}$ denotes the pullback of $\mathcal{C}^{\infty}(E)$ to functions on $E^{(2)}$ via $+: E^{(2)} \to E$.
\end{itemize}

Now we are ready for:
\begin{lemma}\label{lemma:plusbracket}
Suppose that the Lie algebroid structure on $D \to E$ is $q$-projectible, so there is an induced Lie algebroid structure on $D^{(2)} \to E^{(2)}$.  The addition map $+_A$ is a Lie algebroid morphism if and only if, for all $X,Y \in \Gamma_\ell(D,B)$ and $\alpha, \beta \in \Gamma_C(D,B)$, 
\begin{enumerate}
\item $([X,Y]_D)^{(2)}$ is $+$-projectible to $[X,Y]_D$,
\item $([X,\alpha]_D)^+$ is $+$-projectible to $[X,\alpha]_D$,
\item $[\alpha,\beta]_D$ = 0.
\end{enumerate}
\end{lemma}
\begin{proof}
The condition that $+_A$ is a Lie algebroid morphism is equivalent to the statement that ``If $\tilde{X}, \tilde{Y} \in \Gamma(D^{(2)},E^{(2)})$ are $+$--projectible to $X,Y \in \Gamma(D,E)$, respectively, then $[\tilde{X}, \tilde{Y} ]_{D^{(2)}}$ is $+$-projectible to $[X,Y]_D$. ''  Call this property $P$.

Let us choose the same local coordinates as in the proof of Lemma \ref{lemma:qbracket}.  The induced coordinates on $D^{(2)}$ are $\{x^i, e^i_1, e^i_2, a^i, c^i_1, c^i_2 \}$. We also introduce coordinates $e^i_{\pm} \defequal \frac{1}{2}(e^i_1 \pm e^i_2)$ and $c^i_{\pm} \defequal \frac{1}{2}(c^i_1 \pm c^i_2)$.  Then $\{ A_i^{(2)}, C_i^+, C_i^- \}$ is the frame of sections of $D^{(2)}$ over ${E^{(2)}}$ dual to the fibre coordinates $\{a^i, c^i_+, c^i_- \}$.  We notice that a section $\tilde{X} \in \Gamma(D^{(2)}, E^{(2)})$ is $+$--projectible to
\begin{equation*} 
X = f^i(x,e)A_i + g^i(x,e)C_i \in \Gamma(D,E)
\end{equation*}
if and only if it is of the form
\begin{equation*}
\tilde{X} = f^i(x,2e^+)A^{(2)}_i + g^i(x,2e^+) C_i^+  + h^i(x,e^+,e^-) C_i^-.
\end{equation*}
Next we notice that generic sections $X^{(2)} \in \Gamma_l^{(2)}$, $\alpha^+ \in \Gamma_c^+$, and $\alpha^- \in \Gamma_c^-$ have, respectively, the form:
\begin{equation*}
\begin{split}
 X^{(2)} & = f^i(x)A^{(2)} + g^i_j(x)e^j_+ C_i^+, \\
 \alpha^+ & = h^i(x) C_i^+,\\
 \alpha^-  & = h^i(x) C_i^-.
\end{split}
\end{equation*}
These coordinate descriptions show that the space of sections of $D^{(2)} \to E^{(2)}$ that are $+$--projectible is exactly
\begin{equation*}
\mathcal{C}^{\infty}(E)^{(2)} \otimes \Gamma_l^{(2)} \; + \; 
\mathcal{C}^{\infty}(E)^{(2)} \otimes \Gamma_c^+ \; + \; 
\mathcal{C}^{\infty}(E^{(2)}) \otimes \Gamma_c^-.
\end{equation*}
From here, a direct computation shows that property P is true in general if and only if it is true for $\tilde{X}$ and $\tilde{Y}$ in $\Gamma_l^{2} \cup \Gamma_c^+ \cup \Gamma_c^-$.  (This requires using the compatibility conditions for the anchor.)  We are left with six particular cases of the statement of property $P$.

Finally, the bracket on $D^{(2)}$ satisfies the property that, if $(X, X')$ and and $(Y, Y')$ are compatible pairs of sections of $D \to E$, then $[X \times X', Y \times Y']_{D^{(2)}} = [X,Y]_D \times [X',Y']_D$.  We systematically apply this fact to property $P$ in the six particular cases we have, and we obtain conditions (1)--(3) in the statement of this lemma, hence completing its proof.
\end{proof}

The compatibility conditions involving the brackets in the definition of $\LA$--vector bundle have been rewritten in terms of linear and core sections as conditions (1)-(3) in Lemmas \ref{lemma:qbracket} and \ref{lemma:plusbracket}.  The compatibility conditions involving the brackets in the definition of $\VB$--algebroid were conditions (1)-(3) in Definition \ref{dfn:vbla}.  Lemma \ref{lemma:lincore} below shows that the two sets of conditions are equivalent, by means of a characterization of linear and core sections in terms of $q$- and $+$-projectibility, hence completing the proof of Theorem \ref{thm:lavbvbla}.

\begin{lemma} \label{lemma:lincore} \ 
\begin{enumerate}
\item A section $X \in \Gamma(D,E)$ that is $q$-projectible to $X_0 \in \Gamma(A)$ is linear if and only if $X^{(2)}$ is $+$-projectible to $X$.
\item A section $\alpha \in \Gamma(D,E)$ that is $q$-projectible to $0^A$ is a core section if and only if $\alpha^+$ is $+$-projectible to $\alpha$.
\end{enumerate}
\end{lemma}
\begin{proof}
The proof is a computation in coordinates, or a direct check of the definitions.
\end{proof}


\section{Proofs from \S\ref{sec:cc}}
\label{appendix:charclasses}

We begin with the following lemmas, which consist of straightforward extensions of well-known results in standard Chern-Weil theory.  For all the lemmas, we suppose that $B \to M$ is a Lie algebroid and $\Etot \to M$ is a graded vector bundle.

\begin{lemma}\label{lemma:strbracket}
For any (nonhomogeneous) $B$-superconnection $\mathcal{O}$ on $\Etot$ and any $\End(\Etot)$-valued $B$-form $\theta$, the operator $[\mathcal{O}, \theta]$ on $\Omega(B) \otimes \Gamma(\Etot)$ is $\Omega(B)$-linear (and therefore may be viewed as an $\End(\Etot)$-valued $B$-form), and
\begin{equation*}
 \str([\mathcal{O}, \theta]) = d_B \str(\theta).
\end{equation*}
\end{lemma}
\begin{proof}
 Locally, choose a homogeneous frame $\{a_i\}$ for $\Etot$, and express $\mathcal{O}$ as $d_B + \eta$, where $\eta$ is an $\End(\Etot)$-valued $B$-form.  The result follows from the fact that $\str([\eta, \theta]) = 0$.
\end{proof}

\begin{lemma} \label{lemma:chern}
Suppose that $\Etot$ is equipped with a metric $g: \mathcal{E} \iso \mathcal{E}^*$, as in \S\ref{sec:cc}.  For any (nonhomogeneous) $B$-superconnection $\mathcal{O}$ on $\Etot$,
  \begin{enumerate}
    \item $d_B \str(\mathcal{O}^{2k}) = 0$  for all $k$.
\item $\str(\mathcal{O}^{2k}) = (-1)^k \str((\presup{g}\mathcal{O})^{2k})$ for all $k$.
  \end{enumerate}
\end{lemma}

\begin{proof}
 For the first statement, we note that $\mathcal{O}^{2k}$ is an $\End(\Etot)$-valued $B$-form, so by Lemma \ref{lemma:strbracket} we have that $d_B \str(\mathcal{O}^{2k}) = \str([\mathcal{O}, \mathcal{O}^{2k}]) = 0$.

For the second statement, it follows from \eqref{eqn:adjoint} that, for any $a \in \Gamma(\Etot)$ and $\varsigma \in \Gamma(\Etot^*)$,
\begin{equation*}
 \langle \mathcal{O}^{2k} a, \varsigma \rangle = - \langle \mathcal{O}^{2k-2} a, (\mathcal{O}^\dagger)^2\varsigma \rangle = (-1)^k \langle a, (\mathcal{O}^\dagger)^{2k}\varsigma \rangle.\qedhere
\end{equation*}
\end{proof}

Next, we present some lemmas regarding superconnections that are built out of pairs and triplets of superconnections.

\begin{lemma} \label{lemma:ftc1}
 Let $\mathcal{O}_1$ and $\mathcal{O}_2$ be (nonhomogeneous) $B$-superconnections on $\Etot$.  As in \eqref{eqn:transg}, let $\mathcal{T}_{\mathcal{O}_1, \mathcal{O}_2}$ be the $(B \times TI)$-superconnection such that 
\begin{equation*}
\mathcal{T}_{\mathcal{O}_1, \mathcal{O}_2}(a) = t \mathcal{O}_1(a) + (1-t) \mathcal{O}_2(a),
\end{equation*}
where $a \in \Gamma(\Etot)$ is viewed as a $t$-independent section of the pullback of $\Etot$ to $M \times I$.  Then
\begin{equation*}
 \int \dee{t} \dee{\dot{t}} \dot{t} \pdiff{}{t} (\mathcal{T}_{\mathcal{O}_1, \mathcal{O}_2})^{2k} = \mathcal{O}_1^{2k} - \mathcal{O}_2^{2k}.
\end{equation*}
\end{lemma}

\begin{proof}
Using the Leibniz rule and the fact that the differential for $B \times TI$ is $d_B + \dot{t}
\pdiff{}{t}$, we compute
\begin{equation*}
 (\mathcal{T}_{\mathcal{O}_1, \mathcal{O}_2})^2 = t^2 \mathcal{O}_1^2 + (1-t)^2 \mathcal{O}_2^2 + t(1-t) [\mathcal{O}_1, \mathcal{O}_2] + \dot{t}(\mathcal{O}_1 - \mathcal{O}_2).
\end{equation*}
By the Fundamental Theorem of Calculus, we have
\begin{equation*}
 \int \dee{t} \pdiff{}{t} (\mathcal{T}_{\mathcal{O}_1, \mathcal{O}_2})^{2k} = \mathcal{O}_1^{2k} - \mathcal{O}_2^{2k} + O(\dot{t}).
\end{equation*}
Finally, we see that
\begin{equation*}
 \int \dee{\dot{t}} \dot{t} \left(\mathcal{O}_1^{2k} - \mathcal{O}_2^{2k} + O(\dot{t})\right) = \mathcal{O}_1^{2k} - \mathcal{O}_2^{2k}.\qedhere
\end{equation*}
\end{proof}

\begin{lemma} \label{lemma:ftc2}
  Let $\mathcal{O}_1$, $\mathcal{O}_2$, and $\mathcal{O}_3$ be (nonhomogeneous) $B$-superconnections on $\Etot$.  let $\mathcal{T}_{\mathcal{O}_1, \mathcal{O}_2, \mathcal{O}_3}$ be the $(B \times TI \times TI')$-superconnection such that 
\begin{equation*}
\mathcal{T}_{\mathcal{O}_1, \mathcal{O}_2, \mathcal{O}_3}(a) = st \mathcal{O}_1(a) + (1-s)t \mathcal{O}_2(a) + (1-t)\mathcal{O}_3(a),
\end{equation*}
where $t$ and $s$ are coordinates on $I$ and $I'$, respectively, and $a \in \Gamma(\Etot)$ is viewed as an $s$- and $t$-independent section of the pullback of $\Etot$ to $M \times I \times I'$.  Then
\begin{equation*}
 \int \dee{t} \dee{\dot{t}} \dot{t} \pdiff{}{t} (\mathcal{T}_{\mathcal{O}_1, \mathcal{O}_2, \mathcal{O}_3})^{2k} = \mathcal{T}_{\mathcal{O}_1, \mathcal{O}_2}^{2k} - \mathcal{O}_3^{2k}
\end{equation*}
and
\begin{equation*}
 \int \dee{s} \dee{\dot{s}} \dot{s} \pdiff{}{s} (\mathcal{T}_{\mathcal{O}_1, \mathcal{O}_2, \mathcal{O}_3})^{2k} = \mathcal{T}_{\mathcal{O}_1, \mathcal{O}_3}^{2k} - \mathcal{T}_{\mathcal{O}_2, \mathcal{O}_3}^{2k}.
\end{equation*}
\end{lemma}
We omit the proof of Lemma \ref{lemma:ftc2}, since it is similar to that of Lemma \ref{lemma:ftc1}.

\begin{proof}[Proof of Proposition \ref{prop:csclosed}]
Let us set $B = A \times TI$ and $\mathcal{O} = \mathcal{T}_{\totaldiff,\presup{g}\totaldiff}$ in part (1) of Lemma \ref{lemma:chern}.  Since  $d_{A \times TI} = d_A + \dot{t} \pdiff{}{t}$, we have that  

\begin{equation*}
   d_A \int \dee{t} \, \dee{\dot{t}} \str \left( (\mathcal{T}_{\totaldiff,\presup{g}\totaldiff})^{2k} \right) = \int \dee{t} \dee{\dot{t}} \dot{t} \pdiff{}{t} \str \left( (\mathcal{T}_{\totaldiff,\presup{g}\totaldiff})^{2k} \right), \\
\end{equation*}
which by Lemma \ref{lemma:ftc1} is $\str(\totaldiff^{2k}) - \str((\presup{g}\totaldiff)^{2k})$.  Since both $\totaldiff$ and $\presup{g}\totaldiff$ are flat, we conclude that $d_A \cs_k^g(\totaldiff) = 0$.
\end{proof}

\begin{proof}[Proof of Lemma \ref{lemma:keven}]
It is clear from the definitions that $\presup{g}\mathcal{T}_{\totaldiff,\presup{g}\totaldiff} = \mathcal{T}_{\presup{g}\totaldiff,\totaldiff}$, so by part (2) of Lemma \ref{lemma:chern} we have that 
\begin{equation*}
  \str \left( (\mathcal{T}_{\totaldiff,\presup{g}\totaldiff})^{2k} \right) = (-1)^k \str \left( (\mathcal{T}_{\presup{g}\totaldiff,\totaldiff})^{2k} \right).
\end{equation*}
  On the other hand, the substitution $u=1-t$ yields the equation
\begin{equation*}
  \int \dee{t} \dee{\dot{t}} \str \left( (\mathcal{T}_{\totaldiff,\presup{g}\totaldiff})^{2k} \right) = - \int \dee{t} \dee{\dot{t}} \str \left( (\mathcal{T}_{\presup{g}\totaldiff,\totaldiff})^{2k} \right),
\end{equation*}
 so we conclude that $\cs^g_k(\totaldiff) = (-1)^{k-1} \cs^g_k(\totaldiff)$, and therefore if $k$ is even we have $\cs^g_k = 0$.
\end{proof}

\begin{proof}[Proof of Proposition \ref{prop:csodd}]
By Lemma \ref{lemma:keven}, we may restrict ourselves to the case where $k$ is odd.

Let $\mathcal{O}$ be a degree $1$ superconnection such that $\presup{g}\mathcal{O} = \mathcal{O}$.  Such an $\mathcal{O}$ may be constructed as a ``block-diagonal'' $A$-superconnection, where the blocks are self-adjoint $A$-connections on $E_i$ for each $i$.  Since $\mathcal{T}_{\totaldiff,\mathcal{O}}$ is homogeneous of degree $1$, it is manifestly the case that $I \defequal \int \dee{t}\dee{\dot{t}} \str \left( (\mathcal{T}_{\totaldiff,\mathcal{O}})^{2k} \right)$ is an element of $\Omega^{2k-1}(A)$.  
To complete this proof we will show that  $2I$ and $\cs_k^g(\totaldiff)$  differ by an exact term.

Since $\presup{g}\mathcal{T}_{\totaldiff,\mathcal{O}} = \mathcal{T}_{\presup{g}\totaldiff,\mathcal{O}}$, we have by part (2) of Lemma \ref{lemma:chern} that 
\begin{equation*}
 \int \dee{t}\dee{\dot{t}} \str \left( (\mathcal{T}_{\totaldiff,\mathcal{O}})^{2k} \right)  = -\int \dee{t}\dee{\dot{t}} \str \left( (\mathcal{T}_{\presup{g}\totaldiff,\mathcal{O}})^{2k} \right).
\end{equation*}

Then
\begin{equation*}
2\int \dee{t}\dee{\dot{t}} \str \left( (\mathcal{T}_{\totaldiff,\mathcal{O}})^{2k} \right) = \int \dee{t}\dee{\dot{t}} \str \left( (\mathcal{T}_{\totaldiff,\mathcal{O}})^{2k} \right) - \int \dee{t}\dee{\dot{t}} \str \left( (\mathcal{T}_{\presup{g}\totaldiff,\mathcal{O}})^{2k} \right),
\end{equation*}
which by Lemma \ref{lemma:ftc2} is
\begin{equation}\label{eqn:dgdo}
\int \dee{t}\dee{\dot{t}} \dee{s} \dee{\dot{s}}\dot{s} \pdiff{}{s} \str \left( (\mathcal{T}_{\totaldiff, \presup{g}\totaldiff, \mathcal{O}})^{2k} \right).
\end{equation}
Since $d_{A \times TI \times TI'} = d_A + \dot{t}\pdiff{}{t} + \dot{s}\pdiff{}{s}$, we have by part (1) of Lemma \ref{lemma:chern} that \eqref{eqn:dgdo} equals
\begin{equation}\label{eqn:tritrans}
  -d_A \int \dee{t}\dee{\dot{t}} \dee{s} \dee{\dot{s}} \str \left( (\mathcal{T}_{\totaldiff, \presup{g}\totaldiff, \mathcal{O}})^{2k} \right) - \int \dee{t}\dee{\dot{t}} \dee{s} \dee{\dot{s}}\dot{t} \pdiff{}{t} \str \left( (\mathcal{T}_{\totaldiff, \presup{g}\totaldiff, \mathcal{O}})^{2k} \right).
\end{equation}
Ignoring the exact term in \eqref{eqn:tritrans}, we see by Lemma \ref{lemma:ftc2} that the second term is
\begin{equation*}
   \int \dee{s} \dee{\dot{s}} \str \left( (\mathcal{T}_{\totaldiff, \presup{g}\totaldiff})^{2k}\right) - \int \dee{s} \dee{\dot{s}} \str \left( \mathcal{O}^{2k}\right).
\end{equation*}
The latter term vanishes since the integrand does not depend on $\dot{s}$, and the first term is  $\cs_k^g(\totaldiff)$.  We conclude that $\cs_k^g(\totaldiff)$ and the ($2k-1$)-form $2\int \dee{t}\dee{\dot{t}} \str \left( (\mathcal{T}_{\totaldiff,\mathcal{O}})^{2k} \right)$ differ by an exact term.
\end{proof}

\begin{proof}[Proof of Proposition \ref{prop:metric}]
Let $g$ and $g'$ be metrics $\mathcal{E} \iso \mathcal{E}^*$, as in \S\ref{sec:cc}.  By an argument similar to that in the proof of Proposition \ref{prop:csodd}, we may see that the equation
\begin{equation}\label{eqn:twometrics}
 \cs^g_k(\totaldiff) - \cs^{g'}_k(\totaldiff) =  \int \dee{s} \dee{\dot{s}} \str \left( (\mathcal{T}_{\presup{g'}\totaldiff, \presup{g}\totaldiff})^{2k} \right)
\end{equation}
holds up to an exact term.  Thus, we need to show that the right hand side of \eqref{eqn:twometrics} is exact.

First, let $\gamma$ be a smooth path of metrics such that $\gamma(0) = g$ and $\gamma(1) = g'$, and for all $r \in [0,1]$ let $\theta_r$ be the degree $1$ $\End(\Etot)$-valued $A$-form defined as
\begin{equation*}
 \theta_r \defequal \presup{\gamma(r)}\totaldiff - \presup{g}\totaldiff.
\end{equation*}
Since $\presup{\gamma(r)}\totaldiff$ is flat for all $r$, we have that
\begin{equation}\label{eqn:thetaflat}
0 = \left(\presup{\gamma(r)}\totaldiff \right)^2 = \left(\presup{g}\totaldiff + \theta_r \right)^2 = [\presup{g}\totaldiff, \theta_r] + \theta_r^2.
\end{equation}

For an $s$-independent section $a$, we have
\begin{equation*}
 \mathcal{T}_{\presup{\gamma(r)}\totaldiff, \presup{g}\totaldiff}(a) = s \theta_r(a) + \presup{g}\totaldiff(a),
\end{equation*}
so, using the Leibniz rule for superconnections, we can compute
\begin{equation}\label{eqn:teesquared}
\begin{split}
   \left(\mathcal{T}_{\presup{\gamma(r)}\totaldiff, \presup{g}\totaldiff}\right)^2 &= s^2 \theta_r^2 + s[\presup{g}\totaldiff, \theta_r] + \dot{s}\theta_r\\
&= (s^2 - s)\theta_r^2 + \dot{s}\theta_r.
\end{split}
\end{equation}
In the last step of \eqref{eqn:teesquared} we have used \eqref{eqn:thetaflat}.  Thus we see that, up to a constant factor,
\begin{equation*}
 \int \dee{s} \dee{\dot{s}} \str \left( (\mathcal{T}_{\presup{\gamma(r)}\totaldiff, \presup{g}\totaldiff})^{2k} \right) = \str \left(\theta_r^{2k-1}\right).
\end{equation*}
We conclude that the right hand side of \eqref{eqn:twometrics}, which we are trying to prove is exact, equals $\int \dee{r} \pdiff{}{r} \str(\theta_r^{2k-1})$ up to a constant factor. 

Second, let $u_r \in \End(\Etot)$ be defined by the property
\begin{equation*}
 \langle s, s'\rangle_{\gamma(r)} = \langle u_r(s), s' \rangle_g.
\end{equation*}
It may be directly checked that $\presup{\gamma(r)}\totaldiff = u_r^{-1} \circ \presup{g}\totaldiff \circ u_r$.  We then see that 
\begin{equation}\label{eqn:dthetadr}
 \begin{split}
  \pdiff{\theta_r}{r} = \pdiff{}{r}\left[\presup{\gamma(r)}\totaldiff \right] &= \pdiff{u_r^{-1}}{r} \circ \presup{g}\totaldiff \circ u_r + u_r^{-1} \circ \presup{g}\totaldiff \circ \pdiff{u_r}{r}\\
&= \pdiff{u_r^{-1}}{r} u_r \circ \presup{\gamma(r)}\totaldiff + \presup{\gamma(r)}\totaldiff \circ u_r^{-1} \pdiff{u_r}{r}\\
&= \left[ \presup{\gamma(r)}\totaldiff, u_r^{-1} \pdiff{u_r}{r} \right].
 \end{split}
\end{equation}
In the last line of \eqref{eqn:dthetadr}, we have used the identity 
\begin{equation*}
 \pdiff{u_r^{-1}}{r}u_r + u_r^{-1} \pdiff{u_r}{r} = 0.
\end{equation*}
Using the property $\left[ \presup{\gamma(r)}\totaldiff, \theta_r^2 \right] = 0$, which follows from \eqref{eqn:thetaflat}, we deduce that
\begin{equation*}
 \pdiff{\theta_r}{r} \theta_r^{2k-2} = \left[ \presup{\gamma(r)}\totaldiff, u_r^{-1} \pdiff{u_r}{r} \theta_r^{2k-2}\right].
\end{equation*}

Finally, we see that 
\begin{equation*}
\begin{split}
  \pdiff{}{r} \str(\theta_r^{2k-1}) &= (2k-1) \str\left(\pdiff{\theta_r}{r} \theta_r^{2k-2}\right)\\
&= (2k-1) d_A \str \left(u_r^{-1} \pdiff{u_r}{r}\theta_r^{2k-2}\right),
\end{split}
\end{equation*}
where in the last line we have used Lemma \ref{lemma:strbracket}.  Thus we conclude that 
 $\int \dee{r} \pdiff{}{r} \str(\theta_r^{2k-1})$ is exact, which is what we wanted to prove.
\end{proof}

\begin{proof}[Proof of Theorem \ref{thm:cslift}]
Let $\totaldiff$ and $\nue{\totaldiff}$ be the superconnections arising from two horizontal lifts.

As we saw in the proof of Proposition \ref{prop:csodd}, the cohomology class of $\cs^g_k(\totaldiff)$ equals that of $2\int \dee{t}\dee{\dot{t}} \str \left(( \mathcal{T}_{\totaldiff, \mathcal{O}})^{2k}\right)$, where $\mathcal{O}$ is self-adjoint.  Therefore, up to an exact term, $\frac{1}{2} \left(\cs^g_k(\totaldiff) - \cs^g_k(\nue{\totaldiff}) \right)$ is
\begin{equation}\label{eqn:nuesame1}
 \int \dee{t}\dee{\dot{t}} \str \left(( \mathcal{T}_{\totaldiff, \mathcal{O}})^{2k}\right) - \int \dee{t}\dee{\dot{t}} \str \left(( \mathcal{T}_{\nue{\totaldiff}, \mathcal{O}})^{2k}\right).
\end{equation}
Again using an argument from the proof of Proposition \ref{prop:csodd}, we have that, up to an exact term, \eqref{eqn:nuesame1} is
\begin{equation}\label{eqn:dnued}
 \int \dee{s}\dee{\dot{s}} \str \left(( \mathcal{T}_{\totaldiff, \nue{\totaldiff}})^{2k}\right).
\end{equation}

Let $\sigma$ be defined as in \eqref{eq:2hats}, and let us defined a path $u_r$ of automorphisms of $\Omega(A) \otimes \Gamma(C[1] \oplus E)$ by $u_r= 1 + r \sigma$.  Then, by setting $\totaldiff_r \defequal u_r^{-1} \circ \totaldiff \circ u_r$, Theorem \ref{thm:vbtoflat} shows that $\totaldiff_0 = \totaldiff$ and $\totaldiff_1 = \nue{\totaldiff}$.
Thus, by the same argument as in the proof of Proposition \ref{prop:metric}, we see that \eqref{eqn:dnued} is exact.
\end{proof}

\bibliographystyle{abbrv}
\bibliography{bibio}

\def\polhk\#1{\setbox0=\hbox{\#1}{{\o}oalign{\hidewidth
  {\l}ower1.5ex\hbox{`}\hidewidth\crcr\unhbox0}}} \def\cprime{$'$}
\begin{thebibliography}{10}

\bibitem{ari:rhla}
C.~{Arias Abad} and M.~Crainic.
\newblock Representations up to homotopy of {L}ie algebroids, 2009.
\newblock arXiv: 0901.0319.

\bibitem{bott:cc}
R.~Bott.
\newblock Lectures on characteristic classes and foliations.
\newblock In {\em Lectures on algebraic and differential topology (Second Latin
  American School in Math., Mexico City, 1971)}, pages 1--94. Lecture Notes in
  Math., Vol. 279. Springer, Berlin, 1972.
\newblock Notes by Lawrence Conlon, with two appendices by J. Stasheff.

\bibitem{crainic:vanest}
M.~Crainic.
\newblock Differentiable and algebroid cohomology, van {E}st isomorphisms, and
  characteristic classes.
\newblock {\em Comment. Math. Helv.}, 78(4):681--721, 2003.

\bibitem{cf}
M.~Crainic and R.~L. Fernandes.
\newblock Secondary characteristic classes of {L}ie algebroids.
\newblock In {\em Quantum field theory and noncommutative geometry}, volume 662
  of {\em Lecture Notes in Phys.}, pages 157--176. Springer, Berlin, 2005.

\bibitem{elw}
S.~Evens, J.-H. Lu, and A.~Weinstein.
\newblock Transverse measures, the modular class and a cohomology pairing for
  {L}ie algebroids.
\newblock {\em Quart. J. Math. Oxford Ser. (2)}, 50(200):417--436, 1999.

\bibitem{fernandes}
R.~L. Fernandes.
\newblock Lie algebroids, holonomy and characteristic classes.
\newblock {\em Adv. Math.}, 170(1):119--179, 2002.

\bibitem{grabowski}
J.~Grabowski and M.~Rotkiewicz.
\newblock Higher vector bundles and multi-graded symplectic manifolds, 2007.
\newblock arXiv: math/0702772.

\bibitem{higgins-mac}
P.~J. Higgins and K.~Mackenzie.
\newblock Algebraic constructions in the category of {L}ie algebroids.
\newblock {\em J. Algebra}, 129(1):194--230, 1990.

\bibitem{kon-urb}
K.~Konieczna and P.~Urba{\'n}ski.
\newblock Double vector bundles and duality.
\newblock {\em Arch. Math. (Brno)}, 35(1):59--95, 1999.

\bibitem{cwm}
S.~{Mac Lane}.
\newblock {\em Categories for the working mathematician}, volume~5 of {\em
  Graduate Texts in Mathematics}.
\newblock Springer-Verlag, New York, second edition, 1998.

\bibitem{mac:dbl2}
K.~C.~H. Mackenzie.
\newblock {Double Lie algebroids and the double of a Lie bialgebroid}.
\newblock arXiv: math.DG/9808081.

\bibitem{mac:dblie2}
K.~C.~H. Mackenzie.
\newblock Double {L}ie algebroids and second-order geometry. {II}.
\newblock {\em Adv. Math.}, 154(1):46--75, 2000.

\bibitem{mackenzie-2000}
K.~C.~H. Mackenzie.
\newblock Notions of double for {L}ie algebroids, 2000.
\newblock arXiv: math/0011212.

\bibitem{mac:duality}
K.~C.~H. Mackenzie.
\newblock Duality and triple structures.
\newblock In {\em The breadth of symplectic and Poisson geometry}, volume 232
  of {\em Progr. Math.}, pages 455--481. Birkh{\"a}user Boston, Boston, MA,
  2005.

\bibitem{mackenzie-2006}
K.~C.~H. Mackenzie.
\newblock Ehresmann doubles and {D}rinfel'd doubles for {L}ie algebroids and
  {L}ie bialgebroids, 2006.
\newblock arXiv: math/0611799.

\bibitem{mythesis}
R.~A. Mehta.
\newblock {\em Supergroupoids, double structures, and equivariant cohomology}.
\newblock PhD thesis, University of California, Berkeley, 2006.
\newblock arXiv: math.DG/0605356.

\bibitem{pradines}
J.~Pradines.
\newblock Repr{\'e}sentation des jets non holonomes par des morphismes
  vectoriels doubles soud{\'e}s.
\newblock {\em C. R. Acad. Sci. Paris S{\'e}r. A}, 278:1523--1526, 1974.

\bibitem{quillen:superconnections}
D.~Quillen.
\newblock Superconnections and the {C}hern character.
\newblock {\em Topology}, 24(1):89--95, 1985.

\bibitem{vaintrob}
A.~Y. Va{\u\i}ntrob.
\newblock Lie algebroids and homological vector fields.
\newblock {\em Uspekhi Mat. Nauk}, 52(2(314)):161--162, 1997.

\end{thebibliography}
\end{document}